\documentclass[12pt]{article}
\pdfoutput = 1

\usepackage{amsfonts,amsmath,amsthm,amssymb,fancyhdr,float,caption,mathtools}
\theoremstyle{plain}
\newtheorem{theorem}{Theorem}[section]
\newtheorem{prop}[theorem]{Proposition}

\newtheorem{lemma}[theorem]{Lemma}
\newtheorem{cor}[theorem]{Corollary}

\theoremstyle{definition}
\def\mathclap#1{\text{\hbox to 0pt{\hss$\mathsurround=0pt#1$\hss}}}

\newcommand{\la}{\lambda}

\newcommand{\ga}{\gamma}

\newcommand{\om}{\omega}

\newcommand{\ep}{\epsilon} 
\newcommand{\De}{\Delta}
\newcommand{\de}{\delta}
\newcommand{\si}{\sigma}

\newcommand{\R}{{\mathbb R}}
\newcommand{\Z}{{\mathbb Z}}

\newcommand{\T}{{\mathbb T}}

\newcommand{\abs}[1]{\left|#1\right|}

\renewcommand{\geq}{\geqslant}
\renewcommand{\leq}{\leqslant}

\newcommand{\tex}[1]{\texorpdfstring{#1}{?}}
\newcommand{\imm}{\hookrightarrow}
\newenvironment{remark}{\refstepcounter{theorem}\par\medskip\noindent{\it
Remark~\thetheorem.}}{\unskip\nobreak\hfill\hbox{ $\oslash$}\par\medskip}

\newenvironment{definition}{\refstepcounter{theorem}\par\medskip\noindent{\it
Definition~\thetheorem.}}{\unskip\nobreak\hfill\par\medskip}

\usepackage{amsfonts,amsmath,amsthm,amssymb,fancyhdr,float,caption,mathtools,enumitem,nicefrac,verbatim,graphicx,tikz-cd,wasysym}
\usepackage{hyperref} %to get color
\hypersetup{pdfencoding=auto, unicode,breaklinks=true,colorlinks=true}
\usepackage[all]{xy} %diagrams
\usepackage{marginnote}

\newcommand{\vect}[2]{\begin{pmatrix}#1\\#2\end{pmatrix}}
\newcommand{\matr}[4]{\begin{pmatrix}#1&#2\\#3&#4\end{pmatrix}}
\newcommand{\sltz}{{{\rm SL}_2(\Z)}}
\newcommand{\psltz}{{{\rm PSL}_2(\Z)}}
\newcommand{\sltr}{{{\rm SL}_2(\R)}}

\newcommand{\helixmap}{\mathrm{hlx}}
\newcommand{\semitoric}{\mathcal{S}_{\mathrm{ST}}}
\newcommand{\helixspace}{\mathcal{S}_{\mathrm{H}}}
\usepackage[mathcal]{euscript}
\newcommand{\STgroup}{\Z*\Z}

\title{Minimal models in semitoric geometry}

\usepackage{geometry}
\usepackage{lmodern}

\usepackage{sectsty}
\sectionfont{\bfseries\upshape\large}
\subsectionfont{\bfseries\upshape\normalsize}
\subsubsectionfont{\bfseries\upshape\small}

\usepackage{color}
\definecolor{darkgreen}{rgb}{.1,.6,.2}
\definecolor{lightgr}{rgb}{.8,.8,.8}
\definecolor{urlcolor}{rgb}{.1,.1,.4}
\definecolor{blackblue}{rgb}{.1,.1,.3}
\definecolor{tocolor}{rgb}{.1,.1,.5}
\definecolor{urlcolor}{rgb}{.2,.2,.6}
\definecolor{linkcolor}{rgb}{.1,.1,.6}
\definecolor{citecolor}{rgb}{.6,.2,.1}
\definecolor{remcolor}{rgb}{.6,.2,.2}
\definecolor{blue}{rgb}{0,0,.6}

\hypersetup{linkcolor=blue, citecolor=citecolor}
\usepackage{lastpage}

\usepackage{fullpage}

\numberwithin{equation}{section}

\date{}

\author{
\noindent
  D.M.~Kane
  \quad
 J.~Palmer \quad
 \'A.~Pelayo}

\newcommand{\addresses}{
	\bigskip

\noindent Daniel M.~Kane \qquad \texttt{dakane@ucsd.edu}\\
	{\footnotesize \textsc{
	Department of Mathematics, University of California, San Diego \\
	9500 Gilman Drive\\
	La Jolla, CA 92093-0112, USA}}

\bigskip
		
\noindent Joseph Palmer \qquad \texttt{j.palmer@rutgers.edu}\\
	{\footnotesize \textsc{
	Department of Mathematics, Rutgers University \\
	Hill Center - Busch Campus\\
	110 Frelinghuysen Road\\
	Piscataway, NJ 08854-8019, USA}}

\bigskip
	
\noindent \'Alvaro Pelayo \qquad\texttt{alpelayo@math.ucsd.edu}\\
	{ \footnotesize\textsc{
	Department of Mathematics, University of California, San Diego \\
	9500 Gilman Drive\\
	La Jolla, CA 92093-0112, USA}}
}

\begin{document}
\title{
{\bf Minimal models of compact symplectic semitoric manifolds}}
\maketitle
\begin{abstract}
A symplectic semitoric manifold is a symplectic $4$\--manifold endowed with a Hamiltonian
$(S^1 \times \mathbb{R})$\--action satisfying certain conditions.
The goal of this paper is to construct a new symplectic invariant of
symplectic semitoric manifolds, the helix, and give applications. 
The helix is a symplectic analogue of the fan of a 
nonsingular complete toric variety in algebraic geometry, that takes
into account the effects of the monodromy
near focus-focus singularities. We give two
applications of the helix: first, we use it to
give a classification of
the minimal models of symplectic semitoric
manifolds, where ``minimal" is in the sense of
not admitting any blowdowns. The second
application is an extension to the compact
case of a well known result of V\~{u} Ng\d{o}c
about the constraints posed on a symplectic
semitoric manifold by the existence of focus-focus
singularities. The helix permits to translate a symplectic 
geometric problem into an algebraic problem,
and the paper describes a method to solve
this type of algebraic problem.
\end{abstract}

%\vspace{-2em}\tableofcontents
%\vspace{2em}
%\vfill

\section{Introduction}

The revolution in symplectic toric geometry started in the 1980s with the proof of the convexity of the image of the momentum map $F=(f_1,\ldots,f_n) \colon (M,\omega) \to \mathbb{R}^n$ associated to a compact symplectic $2n$\--manifold acted upon by a Hamiltonian $n$\--dimensional compact connected abelian Lie group $T$ (i.e. an $n$\--dimensional torus $T=(S^1)^n$),  due independently to
Guillemin-Sternberg~\cite{GuSt1982} and Atiyah~\cite{At1982}; such manifolds are called \emph{symplectic toric}. In fact, $F(M)$ is the polytope $\Delta$ equal to the convex hull of the image under $F$ of the fixed points of the $T$\--action.   

Shortly after, Delzant proved~\cite{De1988} that $\Delta$ encodes all of the information about the manifold $M$, the form $\omega$,  and the $\omega$-preserving $T$\--action. That is, $\Delta$ is the only symplectic $T$\--equivariant invariant of $(M,\omega,F)$.
He moreover showed that any simple, rational, smooth polytope $\Delta$ arises
as the image of a momentum map of a symplectic-toric manifold; following Guillemin these polytopes are now called \emph{Delzant}.

The existence of this action poses restrictions on $(M,\omega)$ and $F$. For instance, $F$
only has elliptic singularities; moreover, the fibers are tori of dimension $0$ up to $n$ 
(in particular, they are submanifolds of $M$).

Delzant's classification was extended in~\cite{PeVNinvent2009, PeVNacta2011}
to compact and noncompact symplectic $4$\--manifolds acted up by the noncompact Lie group $S^1 \times \mathbb{R}$, under 
certain assumptions (all singularities must be non-degenerate, with none of hyperbolic type, the moment map of the $S^1$\--action
must be proper, and
each fiber contains at most one isolated singularity) these manifolds are called \emph{symplectic semitoric}, and so far are classified
when $M$ is $4$\--dimensional.  In this case the momentum map of the
$(S^1 \times \mathbb{R})$\--action is $F=(f_1,f_2)$, where the Hamiltonian vector field $\mathcal{X}_{f_1}$ is
periodic, but not necessarily $\mathcal{X}_{f_2}$.  The main novelty with respect to symplectic toric
manifolds is that $F$ may have, in addition to elliptic singularities, focus\--focus singularities. The fiber
containing a focus\--focus singularity is not a submanifold, it is homeomorphic to a sphere with it south
and north poles identified (i.e. a torus pinched at the focus\--focus singularity).

Symplectic semitoric manifolds are characterized by  five invariants, one of which is a polygon $P$ constructed from $F(M)$ according to V\~{u} Ng\d{o}c~\cite{VN2007}, by unfolding the singular affine structure induced by $F$ on $F(M)$ as a subset of 
$\mathbb{R}^2$ (in fact $F(M)$ need not even be convex\footnote{and in all important examples it is never a polygon,
including the coupled spin\--oscillator and the spin\--orbit system}).   The 
other four invariants account for the effect of the focus\--focus singularities and the monodromy around them (a fundamental phenomena studied by Duistermaat~\cite{Du1980}), they are:  the number of focus\--focus singularities; a Taylor series in two variables characterizing the dynamics near the fibers containing focus\--focus singularities; a height invariant measuring the volume of certain submanifolds at these singularities; and an index which measures the twist of $F$ viewed as a singular Lagrangian fibration, near focus-focus values of $F$ relative to the global toric momentum map which is used to create the polygon invariant.  

\begin{definition}
A  symplectic toric or symplectic semitoric manifold is \emph{minimal} if it does not admit a blow down.
\end{definition}

For a symplectic toric manifold chopping off a corner of $\Delta$ corresponds
to $T$\--equivariantly blowing up $M$ at a $T$\--fixed point, and the inverse
operation corresponds to blowing down.  
To $\Delta$ one  can associate a \emph{fan}, the one corresponding 
to $(M,\omega)$ when viewed as a nonsingular complete toric variety (the explicit relation appears in~\cite{DuPe2009}). Because of this correspondence the search for their minimal model is reduced to an algebraic problem concerning fans associated to Delzant polytopes. If $2n\geq 6$ the problem is still too difficult but when $2n=4$ Fulton classified the corresponding $2$\--dimensional fans.

\begin{theorem}[W.~Fulton 1993]
The inequivalent minimal models of symplectic toric manifolds are $\mathbb{CP}^2$,
$\mathbb{CP}^1 \times \mathbb{CP}^1$, and a Hirzebruch surface
with parameter $k\neq \pm 1$.
\end{theorem}

The Delzant polytopes of the minimal models are: a simplex ($M=\mathbb{CP}^2$ with any multiple of the Fubini-Study form), a rectangle 
($M=\mathbb{CP}^1 \times \mathbb{CP}^1$ with any product form), and a trapezoid ($M$ a Hirzebruch surface, with its
standard form). The
question is whether Fulton's classification can cover more cases.

\medskip
\noindent
{\bf Main Question}. \emph{What are the inequivalent minimal models of compact symplectic semitoric manifolds?}

\medskip

Even more interesting would be  to know whether the question can be answered as an application of the known invariants.
However, it is not clear what the effect of blowing up and down is
on the known invariants we have just described. 
The image $F(M)$ is no longer necessarily a polygon, or even a convex set.
The polygon $P$ is obtained as the image of a homeomorphism $\varphi \colon F(M) \to P \subset \mathbb{R}^2$ 
which unfolds the singular affine structure of $F(M)$ into $P$, taking into account the monodromy (the construction 
of $\varphi$ is delicate, see~\cite{VN2007}). 
The effect of blowing up or down on $P$ depends on the position of the focus\--focus values of
$F$, and here is where a new invariant of compact symplectic semitoric manifolds we call the \emph{semitoric helix},
denote it by $\mathcal{H}$, 
comes into play.  Like in the toric case, 
$\mathcal{H}$ is given by (an equivalence class of) vectors in $\mathbb{Z}^2$, plus some additional information
which we describe later more precisely and which includes the information of focus\--focus singularities
and monodromy (this does not appear in the toric case).

\begin{figure}
 \centering
 \includegraphics[width = 400pt]{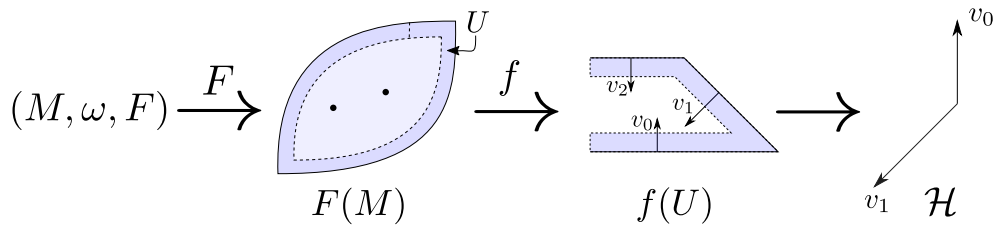}
 \caption{The helix is intrinsically constructed by defining a toric
 momentum map on the preimage of $U$, a neighborhood of the boundary of the image
 of the momentum map minus a single cut, and collecting the
 inwards pointing normal vectors of the piecewise linear boundary
 of the resulting set in $\R^2$.}
 \label{fig_constructhelix}
\end{figure}

Analogous to the way in which from a Delzant polygon one
constructs a fan, from $P$ one constructs the helix $\mathcal{H}$ (after
making some corrections related to the focus-focus singular points), see 
Figure~\ref{fig_helixfrompoly},
though the helix can also be constructed directly from $M$, 
bypassing the polygon, as in Figure~\ref{fig_constructhelix}.
The helix $\mathcal{H}$ 
 contains the information encoding blowing up and blowing down, 
 information which appears very difficult to extract from known 
 invariants. And $\mathcal{H}$ generalizes the fan while taking 
 into account for the effects of the monodromy around the focus-focus 
 singularities.  Moreover, $\mathcal{H}$ can be studied with algebraic 
 techniques, and can be applied to prove:

\medskip
\noindent
{\bf Main Theorem}. \emph{There are precisely seven inequivalent families of minimal models of compact symplectic 
semitoric manifolds.
 Each model is associated to an explicitly describable helix as given in Theorem~\ref{thm_mainintro}.}  
 \medskip
 
One can apply this result to extend a theorem of V{\~u} Ng{\d{o}}c~\cite{VN2007} from
noncompact to compact symplectic semitoric manifolds: if a compact symplectic semitoric manifold 
with momentum map $F=(f_1,f_2)$ has more than two focus-focus singular points, then $f_1$ has either a non-unique 
maximum or a non-unique minimum.  The Main Theorem and application can be used to 
study many integrable systems from classical mechanics such as the 
spin\--orbit system~\cite{LFPe2016, SaZh1999} (see Section~\ref{sec_coupledspin}).
It has precisely one focus\--focus singularity with monodromy, and 
is an example of a compact (nontoric) symplectic semitoric manifold.  
In a different direction, and as an application of Fulton's theorem, 
in~\cite{PePRS2013} some properties of the associated moduli spaces
of manifolds were studied in detail; it would be interesting to use the above result to study a semitoric analogue. 
First steps towards this have been carried out
by the second author in~\cite{PaSTMetric2015},
where a natural topology on the space of symplectic semitoric manifolds
is constructed.

We conclude by explaining the idea of the proof of the Main Theorem. 
The proof of this theorem operates by translating the problem into algebraic language in which a number of things are easier to work with. The basic ideas behind this technique were already present in~\cite{KaPaPe2016}, 
but they are refined and developed so as to be useful
in practical applications. The algebraic correspondence works as follows. To any semitoric helix,
as on the right hand side of Figure~\ref{fig_constructhelix}, there is a natural way to associate a word of a particular form in $\sltz$.
The word is
$
\si = ST^{a_0}\ldots ST^{a_{d-1}},
$
where the $a_i$ are integers and $S,T\in\sltz$ represent the specific matrices
given in Equation~\ref{eqn_ST}.
Our attempt to classify helices will operate by attempting to understand the associated words, which must satisfy two conditions. Firstly, 
$\si$ must be conjugate to $T^c$, where $c\in\Z$ is the
number of focus-focus points of the associated system.
Additionally, we need to ensure that our helix wraps only a single time around the origin before repeating. 
That is, we can define the number of times a collection of vectors
$v_0, \ldots, v_{d-1}$ winds around the origin by following a path
from $v_0$ to $v_0$ which connects $v_i$ to $v_{i+1}$, moves only counterclockwise, and circles the origin the minimal number of times.
In order to detect this winding number from the word
$\si$ in $\sltz$, we will need to lift $\si$ to the universal cover of $\sltr$, in which we can define a function on words which
agrees with the winding number of the associated helix, which we call
the winding number of that word.
We let $G$ denote the preimage of $\sltz$ in the universal cover of $\sltr$. We are then able to produce an exact correspondence between minimal semitoric helices and words of particular form in $G$ that lie in one of a small number of conjugacy classes in $G$.

The key idea in our analysis now follows from the observation that almost all of the appropriate words conjugate to the correct element of $\sltz$ have winding numbers that are too large. In fact, the elements that we are looking for will necessarily have nearly the smallest possible winding number of any word representing the correct element of $\sltz$. In order to properly analyze this, we show that each element of $\psltz$ has a unique minimal word associated to it with the smallest possible winding number. In fact, any representation of the given element can be reduced to the minimal one by means of a few simple reduction steps. The thrust of our argument is now to look at the minimal word associated to the element represented by our helix. Noting that the word corresponding to our helix reduces to this minimal word in only a few steps, allows us to reduce ourselves to a small number of possibilities.
The main novelty of the paper is precisely
this method of proof, which although it may be natural to algebraists, we have not seen 
used in symplectic geometry.

\emph{Acknowledgements.} DMK is supported by NSF grant
DMS-1553288.  AP and JP are supported by
NSF grants DMS-1055897 and DMS-1518420. 
AP also received support from Severo Ochoa Program 
at ICMAT in Spain. Part of this work was carried out at 
Universidad Complutense de Madrid and ICMAT.

\subsubsection*{Structure of the article}
In Section~\ref{sec_mainresults} we state the main results of the paper,
Theorem~\ref{thm_mainintro} and Theorem~\ref{thm_Jmax}.
In Section~\ref{sec_background} we review some background material
regarding symplectic toric manifolds.
In Section~\ref{sec_thehelix} we define the semitoric helix, outline
its construction, and state Theorem~\ref{thm_classifyvectors} which
is a precise version of Theorem~\ref{thm_mainintro}.
In Section~\ref{sec_geometry} we present the construction of
a semitoric helix from a symplectic semitoric manifold.
In Section~\ref{sec_algebraictech} we explain the connection between semitoric
helices and $\sltz$ and
in Section~\ref{sec_pltz} we introduce a standard form for elements
of $\psltz$.
Finally, in Section~\ref{sec_minhelices} we use the results of
Sections~\ref{sec_algebraictech} and~\ref{sec_pltz} to prove 
Theorems~\ref{thm_mainintro} and~\ref{thm_classifyvectors}.
In Section~\ref{sec_examples} we follow the argument of the proof
of Theorem~\ref{thm_classifyvectors} applied to a specific example and also explain the example
of the coupled angular momenta system. Section~\ref{sec_nonminhelices} we prove 
Theorem~\ref{thm_Jmax}.

\section{Main results}
\label{sec_mainresults}

\begin{definition}(\cite{PeVNinvent2009, PeVNacta2011})
 A \emph{symplectic semitoric manifold} is a connected $4$\--dimensional
 integrable system $(M,\om,F = (J,H))$ such that:
 \begin{enumerate}[topsep=0pt, itemsep=0pt]
  \item\label{item_stproper} $J$ is proper, that is, if $K\subset\R$ is compact
  then $J^{-1}(K)$ is compact;
  \item\label{item_stperiodic} the Hamiltonian vector field $\mathcal{X}_J$
   induced by $J$ has periodic flow of period $2\pi$ and the $S^1$\--action
   generated by this flow is effective;
  \item\label{item_stnondeg} all singularities of $F$ are non-degenerate and contain
   no hyperbolic blocks.
 \end{enumerate}
\end{definition}
\vspace{-12pt}
Item~\eqref{item_stnondeg} refers to the Williamson classification
of singularities for integrable systems (see~\cite{Zu1996}).
In~\cite{El1984} Eliasson
extends the pointwise classification of singular points implied
by Williamson's classification of Cartan subalgebras of
$\mathfrak{sp}(2n)$~\cite{Wi1996} to a local normal form for non-degenerate singular
points.
Since $\mathrm{dim}(M)=4$, item~\eqref{item_stnondeg} implies that any point $p\in M$ in
a symplectic semitoric manifold is one of: completely regular; elliptic-regular;
elliptic-elliptic; or focus-focus.
In this article we assume all symplectic semitoric manifolds are
\emph{simple}, which means there is at most one focus-focus
point in each level set of $J$.
\begin{definition}\label{def_stiso}
Given two symplectic semitoric manifolds $(M,\om,F)$ and $(M',\om',F')$
a symplectomorphism $\psi\colon (M,\om)\to (M',\om')$ is a \emph{semitoric isomorphism}
if $\psi^*(J',H') = (J, f(J',H'))$ where $f\colon \R^2\to\R^2$ is a smooth map
with $\frac{\partial}{\partial y}f\neq 0$ everywhere.
\end{definition}
\vspace{-10pt}
\begin{prop}\label{prop_helixisinvrt}
 To each compact symplectic semitoric manifold $(M,\om,F)$
 there is an associated semitoric helix
 denoted $\helixmap(M,\om,F)$
 and $\helixmap(M,\om,F) = \helixmap(M',\om',F')$
 if $(M,\om,F)$ and $(M',\om',F')$ are isomorphic as 
 symplectic semitoric manifolds.
\end{prop}
Proposition~\ref{prop_helixisinvrt} is proven in
Section~\ref{sec_intrinsicconstruction}.
There is a natural notion of blowup/down in this context which
we describe in detail in Section~\ref{sec_blowups}, and a symplectic semitoric
manifold is said to be \emph{minimal} if it does not admit a such a blowdown.
In~\cite{VN2007} it is shown that any symplectic semitoric manifold has only
finitely many focus-focus singular points.

\begin{theorem}\label{thm_mainintro}
 Let $(M,\om,F)$ be a compact symplectic semitoric manifold with $c>0$ focus-focus
 points and $d>0$ elliptic-elliptic points.
 If $d<5$ then $(M,\om,F)$
 is minimal if and only if
 its associated helix is one of:
\begin{center}
 \includegraphics[height = 180pt]{figures/classes}
\end{center} 
 with parameters in each case given by $(1)$ $c=1$; $(2)$ $c=2$; $(3)$ $k\neq\pm 2$, $c=1$; $(4)$ $c\neq 2$; $(5)$ $k\neq \pm 1, 0$, $c\neq 1$; $(6)$ $k\neq -1, 1-c$, $c>0$.
   If $d\geq 5$ then $(M,\om,F)$ is minimal
   if and only if $d>5$ and in this case the associated helix is
   completely determined by $c$ and a positively oriented 
   basis $(v_0, v_1)$ of $\Z^2$.
\end{theorem}

A complete version of Theorem~\ref{thm_mainintro}
appears as Theorem~\ref{thm_classifyvectors}
and an example of a helix with $d>5$ is shown in Figure~\ref{fig_schematic}.
The example of the coupled angular momenta system, which is minimal of 
type~\eqref{type3} with $k=1$, is given in Section~\ref{sec_coupledspin}.
Theorem~\ref{thm_classifyvectors} has the following surprising
consequence in the study of symplectic semitoric manifolds, which
follows from Lemma~\ref{lem_Jmax} and is proven
in Section~\ref{sec_nonminhelices}.

\begin{theorem}\label{thm_Jmax}
 If a compact symplectic semitoric manifold has at least two focus-focus
 singular points and the component of the momentum map with
 periodic flow achieves its maximum and minimum at a single point each, then
 the system must have exactly two focus-focus points and be a minimal
 symplectic semitoric manifold of type~\eqref{type2} from Theorem~\ref{thm_classifyvectors}.
\end{theorem}

In~\cite[Theorem 3]{VN2007} V\~{u} Ng\d{o}c uses an argument related to the
Duistermaat-Heckman measure on symplectic semitoric manifolds to prove
that there do not exist noncompact symplectic semitoric manifolds for 
which the component of the momentum map
with periodic flow achieves its maximum and minimum at a single
point each and which have more than one focus-focus point.
Recently, S.~Sabatini drew to our attention that she had announced 
a version of Theorem~\ref{thm_Jmax}
at a conference in 2013 and outlined a different proof from the one given
in the present paper.

\section{Background: Minimal symplectic toric manifolds}
\label{sec_background}

Here we give a pedestrian exposition of toric manifolds
and integrable systems
from the point of view of symplectic geometry.

An \emph{integrable system} is a triple $(M,\om,F)$ where
$(M,\om)$ is a $2n$\--dimensional symplectic manifold and
$F = (f_1, \ldots, f_n) \colon M\to\R^n$
is a smooth map such that its components $f_1, \ldots, f_n$
Poisson commute and are independent almost
everywhere.  That is, $\om(\mathcal{X}_{f_i}, \mathcal{X}_{f_j}) = 0$
for all $i,j = 1, \ldots, n$ and $(\mathcal{X}_{f_1})_p, \ldots, (\mathcal{X}_{f_1})_p$
are linearly independent in $\mathrm{T}_p M$ for almost all
$p\in M$, where $\mathcal{X}_{f_i}$ denotes the Hamiltonian
vector field of $f_i$.
\begin{definition}
A \emph{symplectic toric manifold}
is an integrable system $(M,\om,F)$ such that
$(M,\om)$ is a compact and connected $2n$\--dimensional
symplectic manifold, each
$\mathcal{X}_{f_i}$ has $2\pi$\--periodic flow,
and the $\T^n$\--action produced by these flows
is effective.
\end{definition}

A convex, compact, rational polygon in $\R^2$ is a \emph{Delzant polygon} if the collection
of inwards-pointing integer normal vectors to the polygon
of minimal length
form what is known as a toric fan.  For vectors $v,w\in\Z^2$ let
$\det(v, w)$ denote the determinant of the matrix with columns
$v, w$.
\begin{definition}
 A \emph{toric fan} of length $d\in\Z_{>0}$ is a collection of vectors
 $(v_0, \ldots, v_{d-1})\in(\Z^2)^d$ such that
 \begin{enumerate}[nosep]
  \item $\det(v_i, v_{i+1}) = 1$ for $i=0, \ldots, d-1$ where $v_d := v_0$;
  \item $v_0, \ldots, v_{d-1}$ are arranged in counter-clockwise
    order.
 \end{enumerate}
\end{definition}
\vspace{-12pt}
Associated to each toric manifold is a toric fan, formed from
the Delzant polygon in this way.

\begin{definition}
If $(v_0, \ldots, v_{d-1})$ is a toric fan of length
$d$ such that
$
 v_i = v_{i-1} + v_{i+1}
$
then a new toric fan of length $d-1$ can be produced by removing $v_i$.  This operation is known as the \emph{blowdown}
and the inverse operation, inserting the sum of two adjacent vectors, is known as a
\emph{blowup}.
\end{definition}

\begin{definition}
A toric fan is \emph{minimal} if
$
 v_i \neq v_{i-1} + v_{i+1}
$
for $i = 0, \ldots, d-1$.
\end{definition}
Minimal toric fans are those on which a blowdown
cannot be performed. A toric fan can be reduced to a minimal toric fan
by performing blowdowns until no more are possible.  On the other hand, this
implies that any toric fan may be obtained from a minimal toric fan by a finite
sequence of blowups.
Minimal toric manifolds are those that do not admit a symplectic toric
blowdown (see Section~\ref{sec_blowups}).
\begin{prop}[\cite{Fulton1993}]
 A blowup/down on a fan corresponds to a blowup/down on the associated
 toric manifold.  In particular, a toric manifold is minimal
 if and only if its fan is minimal.
\end{prop}
Minimal toric fans were classified in~\cite{Fulton1993},
and this implies a classification of minimal toric manifolds.
The group $\sltz$ acts on a toric fan by acting on each
vector in the fan.

\begin{theorem}[Fulton~\cite{Fulton1993}]
\label{thm_toricminimal}
 A toric manifold is minimal if and only if its fan is one of the following up
 to the action of $\sltz$:
 \begin{enumerate}[font=\normalfont]
 \item $v_0 = \vect{1}{0}$, $v_1 = \vect{0}{1}$, $v_2 = \vect{-1}{-1}$\rm{;}%[.5em]
 \item $v_0 = \vect{1}{0}$, $v_1 = \vect{0}{1}$, $v_2 = \vect{-1}{0}$, $v_3 = \vect{0}{-1}$\rm{;}%\\[.5em]
 \item $v_0 = \vect{1}{0}$, $v_1 = \vect{0}{1}$, $v_2 = \vect{-1}{k}$, $v_3 = \vect{0}{-1}$ for $k\in\Z$, $k\neq 0, \pm 1$.
\end{enumerate}
\end{theorem}

\begin{figure}
 \centering
 \includegraphics[height = 120pt]{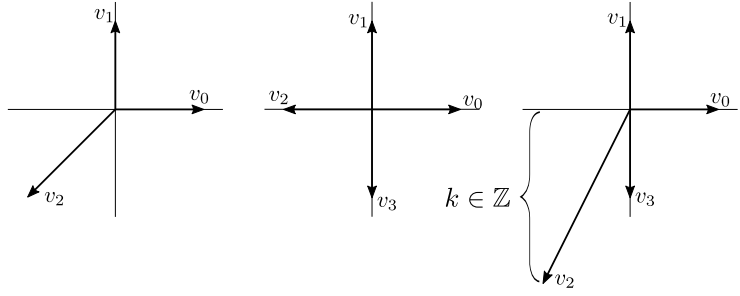}
 \caption{The three possible minimal toric fans
 (up to $\sltz$) listed
 in Theorem~\ref{thm_toricminimal}, where $k\in\Z$
 is the parameter for the Hirzebruch trapezoid and in the figure
 we show the case of $k=-2$.}
 \label{fig_toricfans}
\end{figure}

These fans are shown in Figure~\ref{fig_toricfans}.  
Respectively, these are known as the Delzant triangle, the square, and the Hirzebruch trapezoid named for the shapes of their associated Delzant polygons.
They
correspond, in order, to $\mathbb{C} \mathbb{P}^2$,
$\mathbb{C} \mathbb{P}^1\times\mathbb{C} \mathbb{P}^1$, and a Hirzebruch surface.

\section{The semitoric helix}
\label{sec_thehelix}

\subsection{Toric blowups/downs for symplectic toric and semitoric manifolds}
\label{sec_blowups}

Let $(M,\om,F)$ be a symplectic toric or symplectic semitoric $4$\--manifold with $p\in M$
an elliptic-elliptic point.  Then there exists complex coordinates
$z_1, z_2$ in an open chart $U\subset M$ centered at $p$ such that
the symplectic form is given by
$
 \om_0 = \frac{-\mathrm{i}}{2} \big(\mathrm{d}z_1\wedge\mathrm{d}\overline{z}_1 + \mathrm{d}z_2\wedge\mathrm{d}\overline{z}_2\big)
$
and 
$
 F(z_1, z_2) = \frac{1}{2}\left(\abs{z_1}^2+\abs{z_2}^2\right) + F(0,0).
$
Let $\phi\colon U\to\mathbb{C}^2$ denote the map $\phi = (z_1,z_2)$ and let $V = \phi(U)$.
Let $\mathrm{B}^4(r)\subset\mathbb{C}^2$ denote the standard ball of radius $r>0$.
For any $\la>0$ sufficiently small such that $\mathrm{B}^4(\la)\subset V$
we can define locally in this chart the
toric blowup of weight $\la$.  Since $p$ is an elliptic-elliptic
point this must be possible for some $\la>0$.

Following~\cite[Section 7.1]{McDuffSalamon} define 
$
 \widetilde{\mathbb{C}}^2\subset \mathbb{C}^2\times\mathbb{C}\mathbb{P}^1
$
to be those pairs $(z,\ell)$ such that $z\in\ell$.  That is,
\[
 \widetilde{\mathbb{C}}^2 = \{(z_1, z_2; [w_0, w_1]) \mid w_j z_k = w_{k-1}z_{j+1}\textrm{ for }j = 0,1 \textrm{ and } k = 1,2\}
\]
(the manifold $\widetilde{C}^2$ is the usual (non-symplectic) blowup
of $\mathbb{C}^2$ at the origin).
There are natural projections
\begin{equation*}
 \begin{tikzcd}
   & \widetilde{\mathbb{C}}^2 \ar{dl}[swap]{\pi_{\mathbb{C}^2}}\ar{dr}{\pi_{\mathbb{C}\mathbb{P}^1}}& \\
  \mathbb{C}^2 & & \mathbb{C}\mathbb{P}^1
 \end{tikzcd}
\end{equation*}
and for each $r>0$ define
$
 L(r) = \pi_{\mathbb{C}^2}^{-1}(\mathrm{B}^4(r)).
$
For each $\la>0$ define a symplectic form $\rho(\la)$ on $\widetilde{C}^2$ by
$
 \rho(\la) = \pi_{\mathbb{C}^2}^*\om_0 + \la^2 \pi_{\mathbb{C}\mathbb{P}^1}^*\om_{\mathrm{FS}}
$
where $\om_{\mathrm{FS}}$ is the Fubini-Study form on $\mathbb{C}\mathbb{P}^1$
and $\om_0$ is the standard symplectic form on $\mathbb{C}^2$.
Finally, with $\la$ and $\de$ chosen small enough so that $\mathrm{B}^4(\sqrt{\la^2+\de^2})\subset V$, define
$
 \widetilde{\mathbb{C}^2}_{\la} = \left(\mathbb{C}^2 \setminus B^4\left(\sqrt{\la^2 + \de^2}\right)\right)\cup L(\delta).
$
Since $\rho(\la) = \om_0$ outside of $\mathrm{B}^4(\sqrt{\la^2+\de^2})$ there is no problem
defining a symplectic structure on
$
 \widetilde{M}{(p,\la)} = (M \setminus \phi^{-1}(B^{4}(\sqrt{\la^2+\de^2})))\cap L(\delta),
$
which is known as the \emph{symplectic toric blowup of $M$ at $p$ of size $\la$}.
This is similar to the standard symplectic blowup except that the choice
of chart forces the embedded ball used in this construction to be
$\R^2$\--equivariantly embedded, where the $\R^2$\--action 
on $M$ comes from
the flow of $\mathcal{X}_{F_1}$ and $\mathcal{X}_{F_2}$, which
descends to a $\T^2$\--action for symplectic toric manifolds and an
$(\mathbb{S}^1\times\R)$\--action for symplectic semitoric manifolds
(see~\cite{FiPaPe2016} for an investigation of
symplectic semitoric manifolds as symplectic $(S^1\times\R)$\--manifolds).
We will not show that
this construction is independent of the choices involved because this
is a standard fact (again, see~\cite{KaKePi2007, McDuffSalamon}).

The inverse of this operation
is known as a \emph{toric blowdown}.  Performing a
toric blowup or down on a toric manifold corresponds to
performing a blowup or down on the associated toric fan. 
We will see that performing a toric blowup/down on a symplectic semitoric manifold
corresponds to performing a combinatorial operation, which
we call a blowup/down, on the assocaited
semitoric helix (see Section~\ref{sec_thesemitorichelix}).
We will often simply call a toric blowup a \emph{blowup} (and similar for
a \emph{blowdown}).

\begin{definition}
A symplectic semitoric manifold $(M,\om,F)$ is \emph{minimal}
if it does not admit a blowdown.
\end{definition}

That is, a symplectic semitoric manifold is minimal
if there does not exist any symplectic semitoric manifold
$(M',\om',F')$ such that
$(M,\om,F)$ can be obtained from $(M',\om',F')$ by a symplectic
blowup.

For the present paper we will not be concerned with the
size of the blowups since this will not change the associated
helix and will not effect whether or not the resulting manifold
is minimal.  Thus, we will often say "the blowup of $M$ at $p$"
to really mean "one of the blowups of $M$ at $p$" or even
"the family of all manifolds which can be obtained by performing
a blowup of some weight on $M$ at $p$."

\begin{remark}
 This definition of blowup/down can be extended
 to be used around any
 completely elliptic point of any integrable system
 of any dimension.
\end{remark}

\subsection{The semitoric helix}
\label{sec_thesemitorichelix}

 Let $(\Z^2)^\infty$ be the space of sequences indexed by $\Z$ in $\Z^2$.
 For $\{v_i\}_{i\in\Z}, \{w_i\}_{i\in\Z}\in(\Z^2)^\infty$
 let $\sim$ be the equivalence relation on $(\Z^2)^\infty$
 given by
 $\{v_i\}_{i\in\Z}\sim \{w_i\}_{i\in\Z}$ if and
 only if there exists $k,\ell\in\Z$ such that
 \[v_i = \matr{1}{1}{0}{1}^k w_{i+\ell}\]
  for all $i\in\Z$.
 Let $[\{v_i\}_{i\in\Z}]\in(\Z^2)^\infty/\sim$ denote the equivalence class
 of $\{v_i\}_{i\in\Z}\in (\Z^2)^\infty$.
\begin{definition}\label{def_sthelix}
 A \emph{semitoric helix} is a triple $\mathcal{H} = (d, c, [\{v_i\}_{i\in\Z}])$
 where $d\in\Z_{>0}$, $c\in\Z_{\geq 0}$,
 and $[\{v_i\}_{i\in\Z}] \in (\Z^2)^\infty/\sim$
 such that:
 \begin{enumerate}
  \item $\det(v_i, v_{i+1})=1$ for all $i\in\Z$;
  \item $v_0, \ldots, v_{d-1}$ are arranged in counter-clockwise order;
  \item \label{part_extend} $\matr{1}{c}{0}{1} v_i = v_{i+d}$ for all $i\in\Z$.
 \end{enumerate}
 We say that a semitoric helix $(d,c,[\{v_i\}_{i\in\Z}])$ has \emph{length} $d$
  and \emph{complexity} $c$.
 It is \emph{minimal} if
 \[
  v_{i} \neq v_{i-1} + v_{i+1}
 \]
 for all $i\in\Z$.
\end{definition}

\begin{lemma}
The minimality condition does not depend on the
choice of representative of $[\{v_i\}_{i\in\Z}]$.
\end{lemma}

\begin{proof}
Let
$\{w_i\}_{i\in\Z}\in[\{v_i\}_{i\in\Z}]$ so there exists $k,\ell\in\Z$
such that $v_i = \matr{1}{k}{0}{1} w_{i+\ell}$ for all $i\in\Z$ and thus
$
 v_j = v_{j-1} + v_{j+1}
$
if and only if
$
 w_{j+\ell} = w_{j+\ell-1} + w_{j+\ell+1}
$
by applying $\matr{1}{-k}{0}{1}$.
\end{proof}
A minimal semitoric helix is shown in Figure~\ref{fig_schematic}.
In light of item~\eqref{part_extend}, a semitoric helix of given
complexity $c>0$ and length $d$ is determined by any
$d$ consecutive vectors in any representative.

\begin{definition}
 Let $\mathcal{H} = (d,c,[\{v_i\}_{i\in\Z}])$ be a semitoric helix.
 The \emph{blowup of $\mathcal{H}$ at $v_i$} is the helix 
 $(d+1, c, [\{w_i\}_{i\in\Z}])$,
 where $\{w_i\}$ is formed from $\{v_i\}_{i\in\Z}$ by
 inserting $v_{i+kd} + v_{i+1+kd}$
 between $v_{i+kd}$ and $v_{i+1+kd}$ for all $k\in\Z$.  
 If $v_j=v_{j-1}+v_{j+1}$ for some $j\in\Z$
 then the \emph{blowdown of $\mathcal{H}$
 at $v_i$} is the helix $(d-1,c,[\{u_i\}_{i\in\Z}])$
 where $\{u_i\}$ is produced by
 removing $\{v_{j+nd}\}_{n\in\Z}$ from $\{v_i\}_{i\in\Z}$.
\end{definition}

\begin{figure}
 \centering
 \includegraphics[height = 130pt]{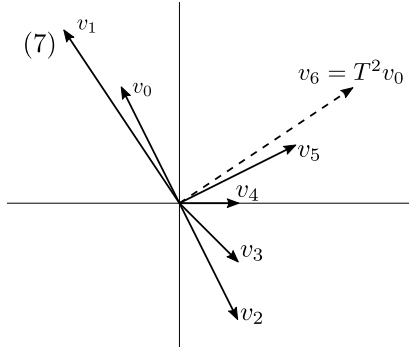}
 \caption{A minimal semitoric helix of length 6
   and complexity 2.
   In the classification from
   Theorem~\ref{thm_classifyvectors} this is a
   type~\eqref{type7} minimal semitoric helix with
   $A_0 = ST^2 ST^2$.}
 \label{fig_schematic}
\end{figure}

\subsection{Outline of helix construction}
\label{sec_helixconstructoutline}

Let $(M,\om,F = (J,H))$ be a compact symplectic semitoric manifold with $c\in\Z$
focus-focus singular points (symplectic semitoric manifolds always have
finitely many focus-focus points~\cite{VN2007}).
Take $U\subset F(M)$ to be a small open neighborhood of the boundary
of $F(M)$ minus a straight line segment $\ell$ which has its endpoints outside
of $U$ and exactly one endpoint in $F(M)$, so that $U$ is simply connected.
The set $F(M)$ is simply connected by~\cite[Theorem 3.4]{VN2007}.
Since $U$ is simply connected the fibers of $F$ form a trivial
torus fibration of $F^{-1}(U)$ so there exists a toric momentum map
$F_{\textrm{toric}}$
on $F^{-1}(U)$
which has the same first component and
associated singular Lagrangian fibration as $F$ (as in~\cite{VN2007}), and
$F_{\mathrm{toric}}(F^{-1}(U))$ minus its interior is
a connected union of line segments in $\R^2$.
Taking the inwards
pointing normal vectors of this segment forms the vectors
$v_0, \ldots, v_d$ and these vectors, along with the integer
$c$, determine the helix of length $d$ associated to $(M,\om,F)$,
which we denote $\helixmap(M,\om,F)$
and which is known as the semitoric helix associated to $(M,\om,F)$.
The precise construction procedure is in Section~\ref{sec_geometry}.

\begin{lemma}\label{lem_sthelix_unique_mfd}
 Given a symplectic semitoric manifold $(M,\om,F)$ there exists exactly
 one associated semitoric helix.
\end{lemma}

Lemma~\ref{lem_sthelix_unique_mfd} is restated with more precision later
as Lemma~\ref{lem_sthelix_unique_mfd2}.

\begin{definition}
Let $\semitoric$ denote the space of symplectic semitoric manifolds
and $\helixspace$ denote the collection of semitoric helices.
 The map
 $\helixmap\colon\semitoric\to\helixspace$
 assigns to each symplectic semitoric manifold $(M,\om,F)$ a semitoric
 helix $\helixmap(M,\om,F)=\mathcal{H}$, where $\mathcal{H}$ is the semitoric
 helix associated to $(M,\om,F)$.
\end{definition}
Lemma~\ref{lem_sthelix_unique_mfd} shows that the map $\helixmap\colon\semitoric\to\helixspace$ is well-defined.

\begin{prop}
 Let $(M,\om,F)$ be a symplectic semitoric manifold with associated
 helix $\helixmap(M,\om,F)$.  The symplectic semitoric manifold $(M',\om',F')$ can be
 obtained from $(M,\om,F)$ by a blowup
 if and only if the associated helix $\helixmap(M,\om,F)$ can be obtained
 from $\helixmap(M,\om,F)$ by a blowup of semitoric helices.
 Moreover, $(M',\om',F')$ can be
 obtained from $(M,\om,F)$ by a blowdown
 if and only if the associated helix $\helixmap(M',\om',F')$ can be obtained
 from $\helixmap(M,\om,F)$ by a blowdown of semitoric helices.
\end{prop}

\begin{proof}
 The helix is obtained as the inwards pointing normal
 vectors on the image of a toric momentum map on a subset
 of $M$, and the blowups we have defined are those which
 produce toric blowups with respect to this momentum map.
 Thus, the correspondence between toric blowups/downs of toric
 manifolds and blowups/downs of toric fans implies the result.
\end{proof}

\subsection{The algebraic technique}

Let $S,T\in\sltz$ be the standard generators given by
 \begin{equation}\label{eqn_ST}
 S = \matr{0}{-1}{1}{0}\textrm{ and } T = \matr{1}{1}{0}{1}
\end{equation}
so
$
 \sltz = \langle S,T \mid STS = T^{-1}ST^{-1}, S^4 = I\rangle
$
and
$
 \psltz = \langle S,T \mid STS = T^{-1}ST^{-1}, S^2 = I\rangle.
$
We denote by $\STgroup$ the free group on
letters $S$ and $T$.

\medskip
\noindent{\bf Notation:} Since we consider several groups generated by $S$ and
$T$ we use $=_H$ to denote equality in the group $H$.
For instance, $S^4 =_{\sltz} I$ but $S^4 \neq_{\STgroup}I$.
\medskip

Given $v,w\in\Z^2$ we denote by $[v,w]$ the $2\times 2$
matrix with first column $v$ and second column $w$ and denote
by $\det(v,w)$ the determinant of $[v,w]$.

The group
$
 G = \langle S,T\mid STS = T^{-1}ST^{-1}\rangle
$
is isomorphic to the preimage of $\sltz$
in the universal cover of
$\sltr$~\cite[Proposition 3.7]{KaPaPe2016}, as
in the following diagram:
\begin{equation*}
 \begin{tikzcd}
  G \ar{d}{}  \ar[hook]{r}{\rho}    &  \widetilde{\sltr}   \ar{d}{} \\
  \sltz \ar{r}{} \ar[hook]{r}{i} &  \sltr
 \end{tikzcd}
\end{equation*}
where $\widetilde{\sltr}$ denotes the universal cover
of $\sltr$, which has fundamental group $\Z$.
Above $\rho\colon G\to \sltr$ denotes the map that takes $G$ isomorphically
to the preimage of $\sltz$ in $\widetilde{\sltr}$ given by
\begin{equation}\label{eqn_rhodef}
 \rho(T) = \matr{1}{t}{0}{1}_{0\leq t \leq 1}\,\,\textrm{ and }\,\,\rho(S) = \begin{pmatrix}\mathrm{cos}\left(\frac{\pi t}{2}\right)&-\mathrm{sin}\left(\frac{\pi t}{2}\right)\\[3pt]\mathrm{sin}\left(\frac{\pi t}{2}\right)&\mathrm{cos}\left(\frac{\pi t}{2}\right)\end{pmatrix}_{0\leq t \leq 1}
\end{equation}
(as in~\cite[Proposition 3.7]{KaPaPe2016}), 
$i\colon \sltz \imm \sltr$ is the inclusion
map and the other two maps are the natural projections.
Each element of the kernel of the natural
projection from $G$ to $\sltz$, denoted
$\mathrm{ker}(G\to\sltz)$, represents a closed loop
in $\sltr$.
Let $(\R^2)^*:=\R^2\setminus\{(0,0)\}$.
\begin{definition}
 Given any closed
loop $\widetilde{\ga}\colon[0,1]\to(\R^2)^*$,
$\widetilde{\ga}(0) = \widetilde{\ga}(1)$, we denote
by $\mathrm{wind}(\widetilde{\ga})\in\Z$ the usual winding number
of $\widetilde{\ga}$ in $(\R^2)^*$.
\end{definition}
Define $\mathrm{pr}\colon\sltr\to(\R^2)^*$ by
\[
 \mathrm{pr}\colon\matr{a}{b}{c}{d}\mapsto\vect{a}{c}.
\]
Since $\pi_1(\sltr) \cong \pi_1((\R^2)*)\cong \Z$ and
$\mathrm{pr}$ sends a generator of $\pi_1(\sltr)$ to a generator
of $\pi_1((\R^2)^*)$, $\mathrm{pr}$ induces an isomorphism
at the level of fundamental groups.
\begin{definition}\label{def_windingofga}
 Given any loop $\ga\colon[0,1]\to\sltr$, $\ga(0) = \ga(1)$, we define
 the \emph{winding number of $\ga$}, denoted $\mathrm{wind}(\ga)$, by
 $\mathrm{wind}(\ga) := \mathrm{wind}\big(\mathrm{pr}(\ga)\big)$.
\end{definition}
In the following section we extend the map
$\mathrm{wind}\circ\rho\colon\mathrm{ker}(G\to\sltz)\to\Z$
to all of $G$.

\subsubsection{The winding number}

Let
$
 W:\STgroup \to \frac{1}{12}\Z
$
be the homomorphism generated by
$
 W(S) = \frac{1}{4}\textrm{ and }W(T) = \frac{-1}{12}.
$
Since $W(STS) = W(T^{-1}ST^{-1})$, $W$
descends to a map on $G$ which we also
denote $W$.  The map is known as the \emph{winding number}~\cite{KaPaPe2016}
because if $\si\in\mathrm{ker}(G\to\sltz)$ then
$W(\si)$ agrees with $\mathrm{wind}(\rho(\si))$
as in Definition~\ref{def_windingofga},
where $\rho$ is as in Equation~\ref{eqn_rhodef}.

 \begin{lemma}[\cite{KaPaPe2016}]
  \label{lem_windandW}
  Given $\si\in\mathrm{ker}(G\to\sltz)$,
  $
   W(\si) = \mathrm{wind}(\mathrm{pr}\circ\rho(\si)).
  $
 \end{lemma}

 \begin{proof}
  The map $W$ is a homomorphism and
  $W(S^4) = \mathrm{wind}(\rho(S^4))$.
  Since $S^4$ is a generator 
  of $\mathrm{ker}(G\to\sltz)\cong \Z$
  this uniquely defines it.
 \end{proof}

For a semitoric helix $\mathcal{H}=(d,c,[\{v_i\}_{i\in\Z}])$ let
$-\mathcal{H} = (d,c,[\{-v_i\}_{i\in\Z}])$.
If $\mathcal{H}=\mathcal{H}'$ or $\mathcal{H}=-\mathcal{H}'$
we write $\mathcal{H}=\pm\mathcal{H}'$.
A cyclic permutation of a list
$(a_0, \ldots, a_{d-1})\in\Z^d$ of integers
is given by
\[(a_{k\,\mathrm{mod}d}, a_{k+1\,\mathrm{mod}d}, \ldots, a_{k+d-1\,\mathrm{mod}d})\]
for some $k\in\Z$.

\begin{prop}\label{prop_correspondence}
 Associated to any
 semitoric helix of length $d$ and complexity $c>0$
 there is a lists of integers $(a_0, \ldots, a_{d-1})\in\Z^d$
 which satisfy
 \[
  ST^{a_0}\ldots ST^{a_{d-1}} =_G S^4X^{-1}T^cX
 \]
 for some $X\in G$.
 This list of integers is unique up to cyclic permutation, and any
 such list of integers is associated to some semitoric helix.
 Semitoric helices $\mathcal{H}$ and $\mathcal{H'}$ have the same length,
 complexity, and associated integers if and only if $\mathcal{H} = \pm\mathcal{H}'$.
\end{prop}
Proposition~\ref{prop_correspondence} is proven in Section~\ref{sec_algebraictech}.

\begin{definition}
A word $\eta\in\STgroup$ is \emph{$S$\--positive} if it can be written using only
non-negative powers of $S$, $T$, and $T^{-1}$.
\end{definition}
To classify the minimal models of symplectic semitoric manifolds
we will show that the associated word of a minimal helix
(as in Proposition~\ref{prop_correspondence}) is very close to the following standard form
in $\psltz$.
\begin{theorem}[Standard form in $\psltz$]
\label{thm_standardpsltz}
 If $X\in\sltz$ there exists a unique string
 $\overline{X}\in\STgroup$
 such that $X=_{\psltz}\overline{X}$ and
 \[
  \overline{X} =_{\STgroup} T^bST^{a_0}\ldots ST^{a_{d-1}}
 \]
 where $a_i>1$ for $i=0, \ldots, d-2$.
 Moreover,
 $
  W(\overline{X})\leq W(\eta)
 $
 for all $S$\--positive $\eta\in\STgroup$
 satisfying $\eta=_{\psltz}X$.
\end{theorem}

We call $\overline{X}$ the \emph{standard form of $X$}.
Theorem~\ref{thm_standardpsltz} is proven in
Section~\ref{sec_pltz}.

\subsection{Main result: minimal models of symplectic semitoric manifolds}

Let
\begin{equation}
 \label{eqn_mathcalS}
 \mathcal{S} = \big\{A\in\sltz \,\big|\, \overline{A} = ST^{a_0}\ldots ST^{a_{d-1}}\textrm{, for }d>5, a_{d-1}\notin\{0,1\}\big\}.
\end{equation}
Recall a semitoric helix of length $d$ is determined by specifying
the complexity and any $d$ consecutive vectors in
any representative of the helix.

\begin{theorem}
 \label{thm_classifyvectors}
 Suppose that $(M,\om,F)$ is a minimal compact symplectic semitoric
 manifold with $c>0$ focus-focus points and associated semitoric helix
 $(d,c,[\{v_i\}_{i\in\Z}]) = \helixmap(M,\om,F)$.
 If $d<5$ then
 the representative $\{v_i\}_{i\in\Z}$ can be chosen to be exactly
 one of the following:
 \newcounter{typecount}
 \setcounter{typecount}{0}
 \begin{center}
 \begin{tabular}{c|c|cl|c}
  \rm{type}  & \rm{length} & $v_0, \ldots, v_{d-1}$ & & \rm{complexity}\\
  \hline
  \refstepcounter{typecount}\rm{(\thetypecount)}\label{type1} &$d=2$  & $\vect{0}{1}$, $\vect{-1}{-2}$ & & $c=1$ \rule{0pt}{4.5ex}\\[3ex]
  \refstepcounter{typecount}\rm{(\thetypecount)}\label{type2} &$d=2$  & $\vect{0}{1}$, $\vect{-1}{-1}$ & & $c=2$ \\[3ex]
  \refstepcounter{typecount}\rm{(\thetypecount)}\label{type3} &$d=3$  & $\vect{0}{1}$, $\vect{-1}{k}$, $\vect{0}{-1}$ & \,\,$k\neq\pm2$& $c=1$ \\[3ex]
  \refstepcounter{typecount}\rm{(\thetypecount)}\label{type4} &$d=3$  & $\vect{1}{0}$, $\vect{0}{1}$, $\vect{-1}{-1}$ & & $c\neq 2$ \\[3ex]
  \refstepcounter{typecount}\rm{(\thetypecount)}\label{type5} &$d=4$  & $\vect{1}{0}$, $\vect{0}{1}$, $\vect{-1}{k}$, $\vect{0}{-1}$ & \,\,$k\neq\pm1,0$ & $c\neq1$ \\[3ex]
  \refstepcounter{typecount}\rm{(\thetypecount)}\label{type6} &$d=4$  & $\vect{1}{0}$, $\vect{0}{1}$, $\vect{-1}{0}$, $\vect{k}{-1}$ & $\begin{array}{l}k\neq -1\\ k\neq1-c\end{array}$ & $c>0$
  \end{tabular}
  \end{center}
 Otherwise, $d\geq 5$, in which case $d>5$,
 and we say that the symplectic semitoric manifold and helix are
 minimal of type
 \refstepcounter{typecount}{\rm (\thetypecount)}\label{type7}.
 There is a one-to-one correspondence between minimal helices of
 type~\eqref{type7} and the set $\mathcal{S}\times\Z_{>0}$.
 Given $c\in\Z_{>0}$ and a basis $v_0,v_1$ of
 $\Z^2$ satisfying $[v_0,v_1]\in\mathcal{S}$ then the corresponding
 minimal helix of type~\eqref{type7} is determined by the following
 procedure:
 Let $a_0, \ldots ,a_{d-1}, d\in\Z$ be the unique
 integers which satisfy
 \begin{equation}\label{eqn_thmclassify}
  S^2 \overline{[v_0,v_1]^{-1}} T^c \overline{[v_0,v_1]} =_{\STgroup} ST^{a_0}\ldots ST^{a_{d-1}}.
 \end{equation}
 Then the recurrence relation
 \[
  v_j = a_{j-2}v_{j-1} + v_{j-2}
 \]
 for $j=0, \ldots, d-1$ and given $v_0,v_1$ determines the 
 vectors $\{v_i\}_{0\leq i<d}$ which, along with the 
 complexity $c$, determine the helix, $\mathcal{H}$.
\end{theorem}

Types~\eqref{type1}-\eqref{type6} are shown in 
Theorem~\ref{thm_mainintro} and
a representative example of type~\eqref{type7} is
shown in Figure~\ref{fig_schematic}.
Theorem~\ref{thm_classifyvectors} is a direct consequence
of Lemma~\ref{lem_classify}.

\begin{cor}\label{cor_formof7}
Suppose that $\mathcal{H}$ is a minimal helix of length
$d>4$.  Then $\mathcal{H}$ is of type~\eqref{type7} from 
Theorem~\ref{thm_classifyvectors} and
there exists a representative $\{v_i\}_{i\in\Z}$
such that $\mathcal{H} = (d,c,[\{v_i\}_{i\in\Z}])$ and the following hold:
\begin{enumerate}[font = \normalfont, nolistsep]
 \item\label{item_formof71} $v_0 = -v_2$;
 \item\label{item_formof72} there exists $k\in\Z$ with $2< k < d$ such that
  $v_{k}$ is $\vect{1}{0}$ or its negative.
\end{enumerate}
\end{cor}

Corollary~\ref{cor_formof7} is proven in Section~\ref{sec_minhelices}.

\subsubsection{Idea of proof of \tex{Theorem~\ref{thm_classifyvectors}}}
In the proof of Theorem~\ref{thm_standardpsltz}, the standard
form in $\psltz$, we use
a reduction algorithm with four steps.  Three of these steps
reduce the winding number by $\nicefrac{1}{2}$ and the remaining step, which corresponds
to a blowdown, does not change the winding number.  We will see,
by Lemmas~\ref{lem_wxpluswxbar} and~\ref{lem_XTXminimal},
that if $a_0, \ldots, a_{d-1}$ is associated to a semitoric helix then
\[
 W(ST^{a_0}\ldots ST^{a_{d-1}}) - W(\overline{ST^{a_0}\ldots ST^{a_{d-1}}}) = \left\{\begin{array}{ll}1 ,&\,X =_{\psltz} T^k\\\frac{1}{2},&\,\textrm{otherwise}\end{array}\right.
\]
and thus we know that
$ST^{a_0}\ldots ST^{a_{d-1}}$ can be reduced to the
standard form from Theorem~\ref{thm_standardpsltz} by using
only one or two of the moves which reduce $W$ along with any number
of blowdowns.

This observation allows us to prove Lemma~\ref{lem_classify}, which classifies
all minimal words satisfying Equation~\eqref{eqn_semitorichelix}.
This implies Theorem~\ref{thm_classifyvectors},
which is proven in Section~\ref{sec_minhelices}.
The method of the proof of Theorem~\ref{thm_classifyvectors}
is carried out on a specific example in
Section~\ref{sec_representativeex}.

\section{From symplectic semitoric manifolds to helices}
\label{sec_geometry}

In this section we give the details of the construction of
the semitoric helix outlined in Section~\ref{sec_helixconstructoutline}.
To do this, we need the following result of  V\~{u} Ng\d{o}c, adapted
slightly to fit the present situation.
\begin{theorem}[{Follows from~\cite[Theorem 3.8]{VN2007}}]
\label{thm_sanstraightening}
 If $(M,\om,F)$ is a symplectic semitoric manifold and $U\subset F(M)$ is
 simply connected, open as a subset of $F(M)$, and contains
 no values of focus-focus points of $F$ then there exists a
 smooth function $f\colon\R^2\to \R^2$ such that
  $f \circ F$ is a momentum map for a Hamiltonian $\T^2$\--action
   on $F^{-1}(U)$ and
  $f$ fixes the first component, i.e.~there exists a function
   $f^{(2)}\colon\R^2\to\R$ such that $f(x,y) = (x, f^{(2)}(x,y))$.
\end{theorem}

Such a function $f\colon\R^2\to\R^2$ is known as a
\emph{straightening map} for the symplectic semitoric manifold
$(M,\om,F)$.

\subsection{Intrinsic construction of the helix}
\label{sec_intrinsicconstruction}

Let $(M,\om,F = (J,H))$ be a compact symplectic semitoric manifold and we
will construct the associated semitoric helix,
$\helixmap(M,\om,F)$.  The images under $F$ of the elliptic-regular
and elliptic-elliptic
singular points all lie in the boundary $\partial F(M)$
and there are finitely many focus-focus points,
whose images lie in the interior $\mathrm{int}(F(M))$
(see~\cite{VN2007}).  Choose a set $U'\subset F(M)$
such that
\begin{enumerate}
 \item $U'$ is open as a subset of $F(M)$;
 \item $U'$ contains $\partial F(M)$;
 \item $U'$ does not contain the image of any focus-focus point;
 \item $U'$ has fundamental group $\Z$.
\end{enumerate}
This is possible because $F(M)$ is simply connected~\cite[Theorem 3.4]{VN2007}
and compact (by assumption).
For instance, $U'$ could be chosen to be the set of all points in $F(M)$ less than
a distance of $\varepsilon$ from the boundary for a sufficiently
small $\varepsilon>0$.
Let $\ell\subset F(M)$ be any line segment starting from a point
in $F(M)\setminus U'$ and ending outside $F(M)$
which intersects $\partial F(M)$ in exactly one connected
component and does not include any singular points of maximal rank of
$F(M)$.  Let $U = U'\setminus\ell$.
We call such a subset a \emph{helix neighborhood}
for $(M,\om,F)$,
see the first step of Figure~\ref{fig_constructhelix}.

By Theorem~\ref{thm_sanstraightening}
there exists a straightening map $f\colon \R^2\to\R^2$ so that
$\mu = f\circ F$ is the momentum map for a Hamiltonian
$\T^2$\--action on $F^{-1}(U)$.
Thus, $f(\partial F(M) \cap U)$ is piecewise linear
of finitely many segments each with rational slope, because
it is the image of the elliptic-regular
and elliptic-elliptic singular points
of $(F^{-1}(U),\om,\mu)$ and this system has only finitely many
elliptic-elliptic fixed points. Let $d\in\Z$ be one less than the
number of segments so
that there are $d+1$ segments in this piecewise linear curve
and let $v_0, \ldots, v_{d}\in\Z^2$ be the
consecutive primitive vectors normal to these
segments facing towards the interior of $f(U)$,
numbered so that $v_0, \ldots, v_{d-1}$
are arranged in counter-clockwise order,
as shown in the last step of Figure~\ref{fig_constructhelix}.

The relationship between $v_0$ and
$v_d$ is determined by the monodromy from
the focus-focus points of the system.
In~\cite{VN2007} V\~{u} Ng\d{o}c studies the monodromy
effect of focus-focus points on toric momentum maps
defined on the preimage of the momentum map image minus a few "cuts"
that remove the focus-focus points and keep the set simply connected.
The proof holds for other simply connected sets,
such as the set $U$, and in this case implies that
$
 v_d = T^c v_0
$
because the set $U$ loops around all $c$ focus-focus points of the system.

Finally, by Definition~\ref{def_sthelix} part~\eqref{part_extend}
$v_0, \ldots, v_d$ extend to a unique semitoric helix 
$\mathcal{H} = (d,c,[\{v_i\}_{i\in\Z}])$.  We say that 
$\mathcal{H}$ is associated to the given symplectic semitoric manifold
$(M,\om,F)$.

Now we must show that the semitoric helix constructed in this
was is the unique one associated to $M$.
That is, we show the helix does not depend on the choices of open
set $U'$, line segment $\ell$, and straightening map $f$ made during the construction.

\begin{lemma}\label{lem_sthelix_unique_mfd2}
 There is precisely one semitoric helix associated to each symplectic semitoric manifold
\end{lemma}

\begin{proof}
 Let $(M,\om,F)$ be a symplectic semitoric manifold with $d$ elliptic-elliptic
 points and $c$ focus-focus points.  Any semitoric helix produced from
 the above construction must have length $d$ and complexity $c$. Now let
 $\mathcal{H}_j = (c,d,[\{v_i^j\}_{i\in\Z}]$ be a semitoric helix
 constructed from $(M,\om,F)$ as above using a set $U_j'$, line segment
 $\ell_j$, and straightening map $f_j$ for $j=1,2$.  We will show 
 $\mathcal{H}_1 = \mathcal{H}_2$.
 
 We may assume that $U_1' = U_2'$ by replacing each with $U' = U_1'\cap U_2'$
 and using the restricted straightening maps. Now
 $
  U_1 = U'\setminus\ell_1\textrm{ and }U_2 = U'\setminus \ell_2
 $
 and, assuming  $U\cap\ell_1 \neq U\cap\ell_2$, the set
 $U_1\cap U_2$ has two connected components
 (if $U\cap\ell_1 = U\cap\ell_2$ the remainder of the proof simplifies).
 Denote these two
 components by $A$ and $B$ ordered so that $v_0^1, \ldots, v_k^1$
 are the inwards pointing normal vectors of the boundary of $f_1(A)$
 and $v_{k+1}^1, \ldots, v_{d}^1$ are the inwards pointing normal
 vectors of $f_1(B)$.
 
 Since $A\subset U_j$ for $j=1,2$ we see 
 $
  f_j|_A\circ F\colon F^{-1}(A)\to\R^2
 $
 is a toric momentum map for $j=1,2$.
 Thus, by \cite[Theorem 3.8]{VN2007}
 there exists $k_A\in\Z$ and $x_A\in\R^2$ such that
 \begin{equation}\label{eqn_sthlxuniquemfdA}
  f_1|_A = T^{k_A}\circ f_2|_A + x_A
 \end{equation}
 and similarly there exists $k_B\in\Z$
 and $x_B\in\R^2$ such that
 \begin{equation}\label{eqn_sthlxuniquemfdB}
  f_1|_B = T^{k_B}\circ f_2|_B + x_B. 
 \end{equation}
 Thus,
 $
  v_i^1 =T^{k_A}v_{i+d-k}^2\,\,\textrm{for }i=0, \ldots, k
 $ 
 and
 $  
  v_i^1 =T^{k_B}v_{i-k-1}^2\,\,\textrm{for }i=k+1, \ldots, d.
 $
 Now, $\{v_i^2\}_{i\in\Z}$ is equivalent in $\Z^2/\sim$ to 
 $\{\widetilde{v}_i^2\}_{i\in\Z}$ defined by
 $
  \widetilde{v}_i^2 = T^{k_A}v_{i+d-k}^2
 $
 and thus
 $
  v_i^1 = \widetilde{v}_i^2 \textrm{ for }i=0, \ldots, k
 $
 and
\[
  v_i^1 = T^{k_B}v_{i-k}^2
        = T^{k_B}T^{-c}v^2_{i-k+d}
        = T^{k_B}T^{-c}T^{-k_A}\widetilde{v}^2_{(i-k+d)-d+k}
        = T^{k_B-k_A-c}\widetilde{v}_i^2
\]
 for $i=k, \ldots, d$
 because $\mathcal{H}_2$ has complexity $c$.
 Thus, $v_i^1 = \widetilde{v}_i^2$ for all $i\in\Z$
 if $k_B-k_A = c$, in which case the proof is complete.
 
 By Equation~\eqref{eqn_sthlxuniquemfdA}
 $f_1$ and $T^{k_A}$ differ by a translation on $A$ so
 $
  f_1|_B = T^c (T^{k_A}\circ f_2)|_B + x_B'
 $
 for some $x_B'\in\R^2$ because this is precisely
 the effect of the monodromy of the set
 $U_1\cap U_2$ encircling all of the $c$
 focus-focus points of the system (see~\cite[Theorem 3.8]{VN2007}).
 Combining this with Equation~\eqref{eqn_sthlxuniquemfdB}
 we see that
 $
  T^c\circ T^{k_A} \circ f_2|_B = T^{k_B}\circ f_2|_B + x_B''
 $
 for some $x_B''\in\R^2$
 and thus $k_B-k_A = c$ as desired.
\end{proof}

Recall Proposition~\ref{prop_helixisinvrt},
that the helix is an invariant of semitoric isomorphism
type.

\begin{proof}[Proof of Proposition~\ref{prop_helixisinvrt}]
 Let $(M,\om,F)$ and $(M',\om',F')$ be symplectic semitoric manifolds
 and let
 $\phi\colon M\to M'$ be a semitoric isomorphism,
 so there exists a smooth map $f\colon\R^2\to\R$
 with $\nicefrac{\partial f}{\partial y}\neq 0$ such
 that $\phi^*F' = (J,f(J,H))$.
 This implies they must
 each have the same number of
 focus-focus points and elliptic-elliptic points.
 Let $d\in\Z$ be the number of elliptic-elliptic points
 and let $c\in\Z$ be the number of focus-focus points.
 Let $U\subset F'(M')$
 be a helix neighborhood for $(M',\om',F')$,
 which is to say it is 
 an open subset that can be used to construct the helix
 associated to $(M',\om',F')$ as is done above, 
 and let $g\colon\R^2\to\R$ be a straightening map for $U$.
 This means there exists some $g^{(2)}\colon\R^2\to\R$
 such that $g(x,y) = (x,g^{(2)}(x,y))$
 and $g\circ F$ is a toric momentum map on $F^{-1}(U)$.
 The semitoric helix associated to $(M',\om',F')$ is
 $\mathcal{H}' = (d,c,[\{v_i\}_{i\in\Z}])$ where $v_0, \ldots, v_{d-1}$
 are the inwards pointing normal vectors of the piecewise
 linear boundary of $g(U)$ and $v_i$ for $i<0$ and $i\geq d$
 is determined by 
 Definition~\ref{def_sthelix} part~\eqref{part_extend} from
 the other vectors and the complexity.
 
 The map $\phi$ descends to the map
 $\widehat{\phi}\colon F(M)\to F'(M')$ given by
 $\widehat{\phi}(x,y) = (x, f(x,y))$ so the following
 diagram commutes:
 \begin{equation*}
 \begin{tikzcd}
    M \ar{r}{\phi}\ar{d}[swap]{F} & M' \ar{d}{F'} \\
   F(M) \ar{r}{\widehat{\phi}} & F'(M')
 \end{tikzcd}
\end{equation*}
The set $\widehat{\phi}^{-1}(U)$ is a helix neighborhood for
$(M,\om,F)$.
Let $\widetilde{g} = g\circ\widehat{\phi}$
and notice that $\widetilde{g}$ is a straightening map
for $\widehat{\phi}^{-1}(U)\subset F(M)$.
Indeed, $\widetilde{g}$
clearly preserves the first component (it is the composition
of maps which preserve the first component) and the second
component of $\widetilde{g}\circ F\colon M\to\R$,
which is 
$g^{(2)}(J, f(J,H))$,
has $2\pi$\--perioidic Hamiltonian flow because 
$g^{(2)}(J, f(J,H)) = \phi^*(g^{(2)}(J',H'))$,
$\phi$ is a symplectomorphism, 
and $g$ is a straightening map for $(M',\om',F')$ so
$g^{(2)}(J',H')$ has $2\pi$\--periodic flow.
The inwards pointing normal vectors of the piecewise
linear portion of the boundary of
$g(\hat{\phi}^{-1}(U))$
generate the helix for $(M,\om,F)$, which we denote $\mathcal{H}$.
Thus, $\mathcal{H} = \mathcal{H}'$ 
because $g(\hat{\phi}^{-1}(U)) = g(U)$,
and since the helix constructed in this
way is unique by Lemma~\ref{lem_sthelix_unique_mfd2}
the helix for $(M,\om,F)$ agrees with
the helix for $(M',\om',F')$.
\end{proof}

To prove that each possible semitoric helix 
is associated to some
symplectic semitoric manifold we need to 
invoke the
classification of symplectic semitoric manifolds, particularly the
semitoric polygon invariant.

\subsection{Delzant semitoric polygons}
\label{sec_Dezsemitoricpoly}

Here we quickly review the definition of a Delzant semitoric polygon
from~\cite{PeVNacta2011} so we can explain the relationship between
semitoric polygons and the semitoric helix.

Let $\pi\colon\R^2\to\R$ denote projection onto the first component
and for any $\la\in\R$ let $\ell_{\la} = \pi^{-1}(\la)$.
A \emph{weighted polygon of complexity $c\in\Z_{\geq 0}$} is a
triple
$
 \De_w = (\De, (\ell_{\la_{j}})_{j=1}^c,(\ep_j)_{j=1}^c)
$
where
\begin{enumerate}[nosep]
 \item $\De\subset\R^2$ is a convex, closed (possibly non-compact), rational
  polygon;
 \item $\ep_j\in\{\pm 1\}$ for $j=1, \ldots, c$;
 \item $\la_j \in \mathrm{int}(\pi(\De))$ for $j=1, \ldots, c$;
 \item $\la_1<\la_2<\ldots <\la_c$.
\end{enumerate}
Let $G_c = \{\pm 1\}^c$ and $\mathcal{G} = \{(T^t)^k : k\in\Z\}$ where
$T^t$ is the transpose of the matrix $T$ given in Equation~\eqref{eqn_ST}.
Given $k\in\Z$ and any vertical line $\ell=\ell_{\la}$, $\la\in\R$, define
 $t_{\ell}^k:\R^2\to\R^2$ by
 \[
  t_{\ell}^k (x,y) = \left\{ \begin{array}{ll}(x,y) ,& x\leq \la\\ (x, k(x-\la)+y) ,& x>\la \end{array}\right.
 \]
and for $\vec{u} = (u_j)_{j=1}^c \in\Z^c$ and $\vec{\la} = (\la_j)_{j-1}^c \in\R$ let
 $
  t_{\vec{u},\vec{\la}}=t^{u_1}_{\ell_{\la_1}} \circ \cdots \circ t^{u_{c}}_{\ell_{\la_c}}.
 $
The group $G_c\times\mathcal{G}$ acts on a weighted polygon by
\[
 ((\ep_j')_{j=1}^c, (T^t)^k) \cdot \left(\De, (\ell_{\la_{j}})_{j=1}^c,(\ep_j)_{j=1}^c\right) = \left(t_{\vec{u},\vec{\la}}\circ (T^t)^k(\De), (\ell_{\la_{j}})_{j=1}^c,(\ep_j' \ep_j)_{j=1}^c\right),
\]
where $\vec{u} = \left(\nicefrac{(\ep_j - \ep_j \ep_j')}{2}\right)_{j=1}^c$.
A weighted polygon is called \emph{admissible} if this
action of $G_c\times\mathcal{G}$ preserves its convexity.  
Let $\mathcal{W}\mathrm{Polyg}_c(\R^2)$
denote the space of all admissible weighted polygons of 
complexity $c\in\Z_{\geq0}$.
An element of $\mathcal{W}\mathrm{Polyg}_c(\R^2)\,/\,G_c\times\mathcal{G}$ is known as
a \emph{semitoric polygon}.

Let
$
 \De_w = (\De, (\ell_{\la_{j}})_{j=1}^c,(\ep_j)_{j=1}^c)
$
be a weighted polygon of complexity $c\in\Z_{\geq0}$.
Let $p\in\De$ be a vertex and let $v,w\in\Z^2$ be the inwards pointing
 normal vectors to the edges adjacent to $p$ of minimal length ordered
 so that $\det(v,w)> 0$.
 Such vectors exist because $\De$ is rational.
 The vertex $p$ satisfies:
 \begin{enumerate}[noitemsep]
  \item the \emph{Delzant condition} if $\det(v,w)=1$;
  \item the \emph{hidden Delzant condition} if $\det(Tv,w) = 1$;
  \item the \emph{fake condition} if $\det(Tv,w) = 0$.
 \end{enumerate}

 Let
 \[
  \partial^{\mathrm{top}}\De = \big\{(x,y): x\in\pi(\De),\, y = \sup\{y_0\in\R:(x,y_0)\in\De\}\big\}
 \]
 denote the top boundary of $\De$.

 \begin{definition}(\cite{PeVNacta2011})
  Let
  $
   [\De_w]\in\mathcal{W}\mathrm{Polyg}_c(\R^2)\,/\,G_c\times\mathcal{G}
  $
  be a semitoric polygon and suppose that $\De_w$ is a representative
  of the form
  $
   \De_w = (\De, (\ell_{\la_{j}})_{j=1}^c,(+1)_{j=1}^c).
  $
  Then $[\De_w]$ is a \emph{Delzant semitoric polygon} if
  \begin{enumerate}[nosep]
   \item $\De\cap\ell_\la$ is either compact or empty for all $\la\in\R$;
   \item each point in $\partial^{\mathrm{top}}\De\cap\ell_{\la_{j}}$ satisfies
    either the hidden Delzant or fake condition, and is hence
    known as a hidden or fake corner, respectively, for $j=1, \ldots, c$;
   \item all other vertices of $\De$ satisfy the Delzant condition and 
    are known as Delzant corners.
  \end{enumerate}
 We say $[\De_w]$ is compact whenever $\De$ is compact for one, and hence all,
 representatives.
 \end{definition}

 Every symplectic semitoric manifold determines a Delzant semitoric polygon~\cite{PeVNacta2011,VN2007},
 which is compact if the manifold is compact.

\subsection{Construction of the helix from the polygon invariant}
\label{sec_polytohelix}

The helix can also be constructed from the V\~{u} Ng\d{o}c polygon
associated to the symplectic semitoric manifold~\cite{VN2007}.
Here we give a brief outline of that construction,
shown in Figure~\ref{fig_helixfrompoly}.
Let $(M,\om,F)$ be a compact symplectic semitoric manifold.

\begin{enumerate}[label = Step \arabic*:, align = left]
 \item{\bf Construct polygon:} Associated to $(M,\om,F)$ is a semitoric
   polygon \[[\De_w] = [(\De,(\ell_{\la_j})_{j=1}^c, (\ep_j)_{j=1}^c)]\] 
   as in~\cite[Theorem 3.9]{VN2007} and described above
   in Section~\ref{sec_Dezsemitoricpoly}.
 \item{\bf Construct semitoric fan:} $\De$ is rational, so take the 
   collection of inwards pointing
   integer normal vectors $w_0, \ldots, w_{m-1}$ of minimal length to its edges
   (this is known as a \emph{semitoric fan}, see~\cite{KaPaPe2016}).
   These can be chosen so that the corner between $w_{m-1}$ and $w_0$
   is not on any of the vertical lines $\ell_{\la_j}$;
 \item{\bf Correct for monodromy effect:} Each consecutive pair of vectors
   $(w_j, w_{j+1})$ is labeled as either fake, hidden, or Delzant
   depending on the type of corner of $\De$ it corresponds to.
   For each $j$ such that $(w_j, w_{j+1})$
   is either hidden or fake replace $w_{j+1}, \ldots, w_{m-1}$
   by
   \[
    Tw_{j+1}, \ldots, Tw_{m-1}.
   \]
   Label the new list of vectors
   $w_0', \ldots, w_{m-1}'$;
 \item{\bf Remove repeated vectors:} Now each pair $(w_i', w_{i+1}')$ 
   either satisfies $\det(w_i', w_{i+1}')=1$
   or $w_i' = w_{i+1}'$.  For each $j$ such that $w_j' = w_{j+1}'$ remove
   $w_{j+1}'$ from the list and when all repeated vectors are
   removed denote the remaining vectors
   by $v_0, \ldots, v_{d-1}$.  Notice $\det(v_i, v_{i+1})=1$ for
   all $i=0, \ldots, d-2$;
 \item{\bf Extend to helix:} By condition~\eqref{part_extend} of Definition~\ref{def_sthelix}
  there exists a unique helix of length $d$ and complexity $c$
  (the number of focus-focus points of the original symplectic semitoric manifold)
  with the given $v_0, \ldots, v_{d-1}$ from the previous step.
\end{enumerate}

\begin{figure}
 \centering
 \includegraphics[width = 400pt]{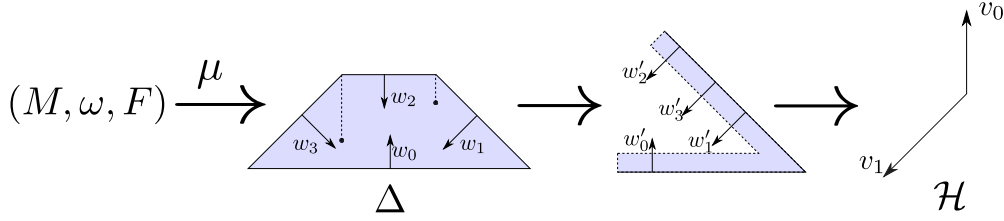}
 \caption{The helix can be recovered from the semitoric polygon
 invariant by "unwinding" the polygon to correct for the effect
 of the focus-focus points and then removing the resulting
 repeated vectors.}
 \label{fig_helixfrompoly}
\end{figure}

\begin{remark}\label{rmk_consecutivefakes}
 Let $[(\De,(\ell_{\la_j}, \ep_j, k_j)_{j=1}^c)]$ be a compact semitoric polygon
 which has no hidden corners and such that all of the fake corners are
 consecutive (while traversing the boundary of $\De$)
 and let $v_0, \ldots, v_{d-1}$
 be the primitive integral inwards pointing normal vectors 
 to every edge of $\De$ which is adjacent to at least one Delzant corner.
 The associated semitoric helix is the unique helix of length
 $d$ and complexity $c$ with the given
 $v_0, \ldots, v_{d-1}$.
\end{remark}

\begin{remark}
 Proposition~\ref{prop_helixisinvrt} also follows from the fact that the
 semitoric polygon invariant is an invariant of the semitoric isomorphism
 type and the above construction of the helix from the semitoric polygon,
 but we have chosen to prove Proposition~\ref{prop_helixisinvrt} in a way
 which is independent of the existence of the semitoric polygon invariant.
\end{remark}

\subsection{Surjectivity of the helix map}

\begin{lemma}\label{lem_helixtopoly}
 Given any semitoric helix $\mathcal{H}$
 there exists a symplectic semitoric manifold $(M,\om,F)$
 such that
 $
  \helixmap(M,\om,F)=\mathcal{H}.
 $
\end{lemma}

\begin{proof}
  Let $\mathcal{H} = (d,c,[\{v_i\}_{i\in\Z}])$ be a semitoric helix.
  Define a collection of vectors
  $w_0, \ldots, w_{d+c}$ by
  $
   w_i = v_i \textrm{ for }i=0, \ldots, d-1
  $
  and
  $
   w_i = T^{i-d+1}v_{d-1}\textrm{ for }i=d, \ldots, d+c.
  $
  Then $\det(w_i, w_{i+1}) = 1$ for $i=0, \ldots, d-1$,
  $
   \det(Tw_i, w_{i+1}) = \det(Tw_i,Tw_i) = 0
  $
  for $i=d, \ldots, d+c-1$, and $w_0 = w_{d+c}$ by
  the periodicity requirement on the helix $\mathcal{H}$.
  The vectors $w_0, \ldots, w_{d+c-1}$ are arranged
  counter-clockwise so there exists a polygon $\De\subset\R^2$
  with $d+c$ edges
  which has these as inwards pointing normal vectors.
  The polygon $\De$ has $d$ Delzant corners
  $c$ fake corners, and since $T$ does not change the
  $y$-value of a vector we see that either all of the fake
  corners are on the top boundary of $\De$ or all of the fake
  corners are on the bottom boundary of $\De$.
  Let $\la_i$ be the horizontal position of the $i^{\textrm{th}}$
  fake corner and we may number these so that
  $\la_1<\la_2<\ldots <\la_c$
  since each vertical line intersects the top and bottom
  boundaries at most once each.  If the fake corners are on
  the top boundary let $\ep_j = +1$ for $j=1, \ldots, c$
  and otherwise let $\ep_j = -1$ for $j=1, \ldots, c$.
  Then,
  $
   \De_w = [(\De, (\ell_{\la_j})_{j=1}^c, (\ep_j)_{j=1}^c)]
  $
  is a Delzant semitoric polygon with
  associated semitoric helix $\mathcal{H}$.
  By~\cite[Theorem 4.6]{PeVNacta2011}
  there exists a symplectic semitoric manifold with $[\De_w]$
  as its semitoric polygon.
 \end{proof}

\begin{remark}
Lemma~\ref{lem_helixtopoly} shows that the map
$\helixmap\colon\semitoric\to\helixspace$
is
surjective by producing a right inverse, but this map
is not injective.
In terms of the Pelayo-V\~{u} Ng\d{o}c invariants this is because
the helix does not encode any information about the Taylor
series invariant, the volume invariant, the twisting index,
the horizontal position of the focus-focus points,
or the lengths of the edges of the semitoric polygon.
\end{remark}

\section{Semitoric helices and \tex{$\sltz$}}
\label{sec_algebraictech}

In this Section we prove Proposition~\ref{prop_correspondence}, which
is the tool we use to translate questions about semitoric helices
into questions about words on letters $S$ and $T$.

\begin{lemma}\label{lem_listofint}
 Given any semitoric helix $\mathcal{H}=(d,c,[\{v_i\}_{i\in\Z}])$
 there exists a list of
 integers $(a_0, \ldots, a_{d-1})\in\Z^d$ such that
 \begin{equation}
  \label{eqn_virecurrence}
  a_{i\, \mathrm{mod} d} v_{i+1} = v_i + v_{i+2}
 \end{equation}
 for all $i\in\Z$.  Furthermore, given $v_0$, $v_1$,
 and $(a_0, \ldots, a_{d-1})$ the helix can be recovered.
\end{lemma}

\begin{proof}
Let $\mathcal{H}=(d,c,[\{v_i\}_{i\in\Z}])$.
Let $A_i = [v_i, v_{i+1}]$ and write $v_{i+2}$ in the $(v_i, v_{i+1})$ basis as
$
 v_{i+2} = b_i v_i + a_i v_{i+1},
$
for $a_i, b_i\in\Z$.
Thus,
\[
 A_{i} \matr{0}{b_i}{1}{a_i}= A_{i+1}
\]
and since $A_i, A_{i+1}\in\sltz$ we see the determinant of each side is 1 so $b_i = -1$ and $v_{i+2}+v_i = a_i v_{i+1}$ as desired. Conversely, given $v_0, v_1$ and $(a_0, \ldots, a_{d-1})$ the
helix can be recovered by using the recurrence relation Equation~\eqref{eqn_virecurrence}.
\end{proof}

\begin{definition}\label{def_associatedints}
The $(a_0, \ldots, a_{d-1})\in\Z^2$ in Lemma~\ref{lem_listofint} are the \emph{associated integers}
to $\mathcal{H}$.
\end{definition}

Recall $\widetilde{\sltr}$ denotes the universal cover of
$\sltr$ with base point at the identity, so $\alpha\in\widetilde{\sltr}$ is a
continuous map $\alpha\colon[0,1]\to\sltr$ satisfying
$\alpha(0) = I$.
The group $G$ is isomorphic to the preimage of
$\sltz$ in $\widetilde{\sltr}$~\cite[Proposition 3.7]{KaPaPe2016} via
the homomorphism $\rho\colon G\to\widetilde{\sltz}$
generated by its action on $S$ and $T$ given in Equation~\ref{eqn_rhodef}.
The operation in $G$ is concatenation of paths.
If $\alpha, \beta\in\widetilde{\sltr}$ then
$\alpha, \beta\colon [0, 1]\to\sltr$
and we define $\alpha\beta\in\widetilde{\sltr}$ by
\[
 \alpha \beta (t) = \left\{\begin{array}{ll}\alpha(2t) ,& 0\leq t \leq \nicefrac{1}{2}\\\alpha(1)\beta(2t-1) ,& \nicefrac{1}{2} \leq t \leq 1.\end{array}\right.
\]
That is, the path $\alpha\beta$ is obtained by traveling
first along the path $\alpha$ and then along the path
produced by multiplying each element of the path $\beta$
on the left by $\alpha(1)$.  
It turns out that the path
produced by traveling first along $\beta$ and then along
$\alpha$ multiplied on the right by $\beta(1)$ is homotopic to $\alpha\beta$.
The next result follows from the fact that the fundamental
group of a topological group is abelian (see~\cite[Section 3.C, Exercise 5]{Hatcher}),
but we prove it here for completeness.

\begin{lemma}\label{lem_homotopy}
 If $\alpha,\beta\in\widetilde{\sltr}$ then the paths
 in $\sltr$ from $I$ to $\alpha(1)\beta(1)$ given by
 \[
  \gamma_0(t) = \left\{\begin{array}{ll}\beta(2t),&\,0\leq t \leq \nicefrac{1}{2}\\\alpha(2t-1)\beta(1),&\,\nicefrac{1}{2}<t\leq 1\end{array}\right.
 \]
 and
 \[
  \gamma_1(t) = \left\{\begin{array}{ll}\alpha(2t),&\,0\leq t \leq \nicefrac{1}{2}\\\alpha(1)\beta(2t-1),&\,\nicefrac{1}{2}<t\leq 1\end{array}\right.
 \]
 are homotopic.

\end{lemma}

\begin{proof}
 A continuous homotopy between them is given by
 \[
  \gamma_s(t) = \left\{\begin{array}{ll}\alpha(2t),&\,0\leq t\leq \nicefrac{s}{2}\\
                                      \alpha(s)\beta(2t-s) ,&\, \nicefrac{s}{2}\leq t \leq \frac{1+s}{2}\\
                                      \alpha(2t-1)\beta(1) ,&\,\frac{1+s}{2}\leq t \leq 1
                       \end{array}
                \right.
 \]
 for $0\leq s \leq 1$, which is shown in Figure~\ref{fig_homotopy}.
 Indeed, $\gamma_s$ is continuous because
 $\gamma_s(\nicefrac{s}{2}) = \alpha(s)$ since $\beta(0) = I$ and
 $\gamma_s(\nicefrac{(1+s)}{2}) = \alpha(s)\beta(1)$.
 It is left to the reader to check that it is a homotopy
 from $\gamma_0$ to $\gamma_1$.
\end{proof}

\begin{figure}
 \centering
 \includegraphics[width=250pt]{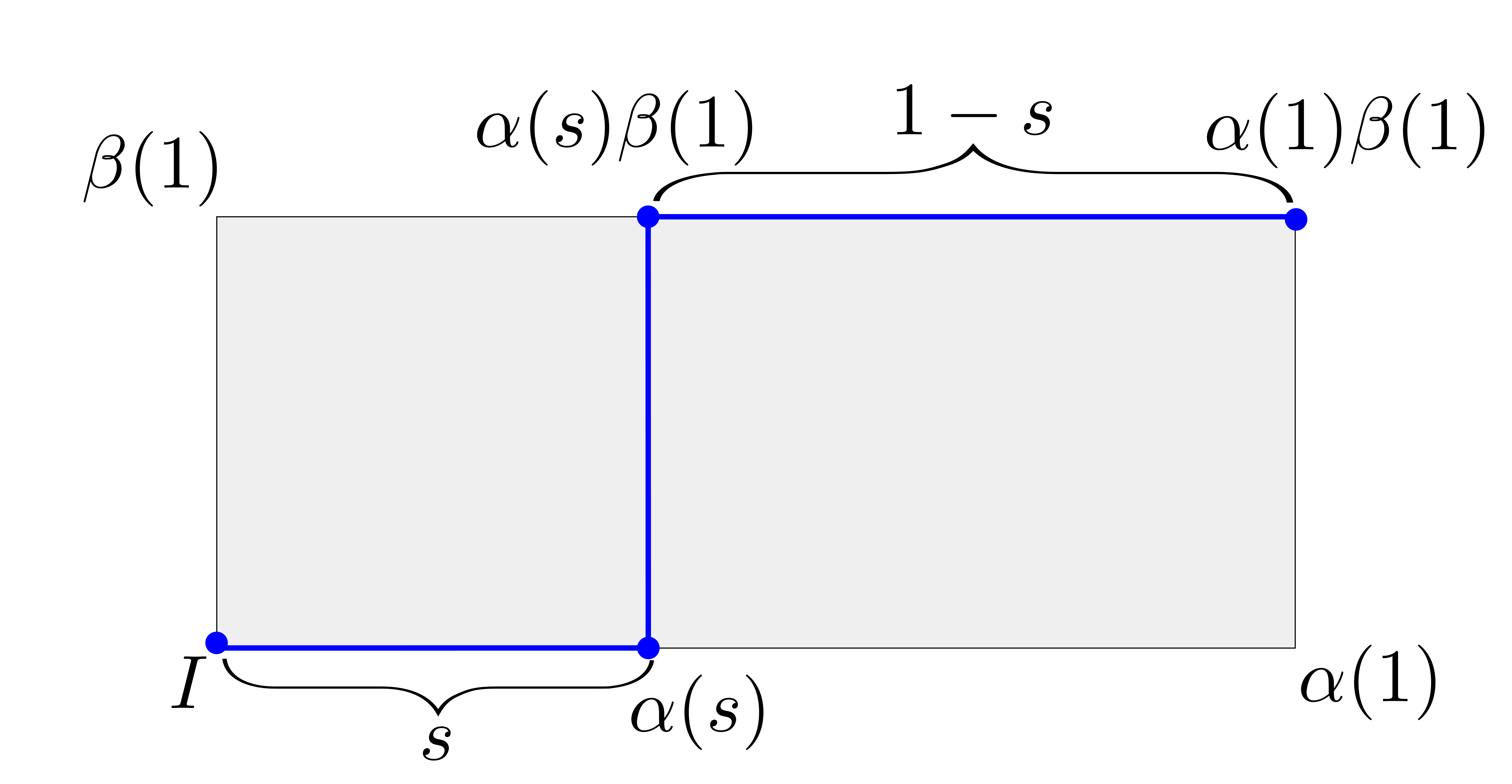}
 \caption{The homotopy from the proof of Lemma~\ref{lem_homotopy}.
 A point $(x,y)$ in the above plane represents $\alpha(x)\beta(y)\in\sltr$.}
 \label{fig_homotopy}
\end{figure}

 Recall the map $\mathrm{pr}\colon \sltr\to(\R^2)^*$, where $(\R^2)^*=\R^2\setminus\{(0,0)\}$, given
 by $\mathrm{pr}([v_1,v_2])=v_1$.  Since
 $
  \pi_1(\sltr) \cong \pi_1((\R^2)^*) \cong \Z
 $
 and
 \[
  \mathrm{pr}\matr{\mathrm{cos}(2\pi t)}{-\mathrm{sin}(2\pi t)}{\mathrm{sin}(2\pi t)}{\mathrm{cos}(2\pi t)} = \vect{\mathrm{cos}(2\pi t)}{\mathrm{sin}(2\pi t)}
 \]
 for $0\leq t\leq 1$ we see $\mathrm{pr}$ sends a generator of $\pi_1(\sltr)$ to a generator of
 $\pi_1((\R^2)^*)$, so $\mathrm{pr}^*\colon\pi_1(\sltr)\to\pi_1((\R^2)^*)$ is an
 isomorphism.

\begin{definition}
Let $\theta\colon(\R^2)^*\to[0,2\pi)$ be the usual angle coordinate
from polar coordinates on $\R^2$ and let $R_\phi\colon\R^2\to\R^2$ be
rotation by the angle $\phi\in[0,2\pi)$.
We say a path $\ga\colon[0,1]\to(\R^2)^*$ \emph{travels counter-clockwise
at most one full rotation} if there exists some $\phi\in[0,2\pi)$
such that $t\mapsto\theta(R_\phi(\ga(t)))$ is an increasing
function for $t\in (0,1)$.
\end{definition}
\begin{lemma}
 \label{lem_SThomotopic}
 Let $\mathcal{H}=(d,c,[\{v_i\}_{i\in\Z}])$ be a semitoric helix
 and let $A_0 = [v_0, v_1]$.
 If $\si\in G$ is given by
 \[
  \si =_G ST^{a_0}\ldots ST^{a_{d-1}}
 \]
 then $\mathrm{pr}\big(A_0\,\rho(\si)\big)$ is homotopic to a path from
 $v_0$ to $v_{d-1}$ which
 travels counter-clockwise at most one full rotation.
\end{lemma}

\begin{proof}
 Let $A_i = [v_i, v_{i+1}]$ for $1\leq i\leq d-1$ and recall $A_{i} = A_{i-1} ST^{a_{i-1}}$.
 Thus,
 \[
  \mathrm{pr}\big(A_{i-1}\,\rho(ST^{a_{i-1}})\big)
 \]
 is a path from $v_i$ to $v_{i+1}$ which is homotopic to
\[
  \ga_i (t) = \mathrm{pr}\left( A_{i-1} \begin{pmatrix} \mathrm{cos}\left(\frac{\pi t}{2}\right) & -\mathrm{sin}\left(\frac{\pi t}{2}\right) + t a_{i-1} \mathrm{cos}\left(\frac{\pi t}{2}\right) \\[3pt] \mathrm{sin}\left(\frac{\pi t}{2}\right) & \mathrm{cos}\left(\frac{\pi t}{2}\right) + t a_{i-a} \mathrm{sin}\left(\frac{\pi t}{2}\right)\end{pmatrix} \right)
            = \mathrm{cos}\left(\frac{\pi t}{2}\right) v_{i-1} + \mathrm{sin}\left(\frac{\pi t}{2}\right) v_i
\]
for $0\leq t\leq 1$.
 The path $\ga_i$ travels only counter-clockwise at most one full rotation from
 $v_{i-1}$ to $v_i$ so the composition of paths $\ga_1, \ldots, \ga_{d-1}$ travels
 counter-clockwise from $v_0$ to $v_{d-1}$.  The result follows because
 $v_0, \ldots, v_{d-1}$ are arranged in counter-clockwise order.
\end{proof}

\begin{lemma}\label{lem_sthelixeqn}
 The integers
 $(a_0, \ldots, a_{d-1})\in\Z^d$ are associated to
 a semitoric helix of complexity $c\geq0$ if and only if
 \begin{equation}\label{eqn_semitorichelix}
  ST^{a_0}\ldots ST^{a_{d-1}} =_{G} S^4 X^{-1}T^c X
 \end{equation}
 for some $X\in G$.
 If $\mathcal{H}=(d,c,[\{v_i\}_{i\in\Z}])$ is a semitoric
 helix with associated integers $(a_0, \ldots, a_{d-1})$ then $A_0 = [v_0,v_1]$
 satisfies $X =_G A_0$.
\end{lemma}
\begin{proof}
 Let $A_i = [v_i, v_{i+1}]$. By Lemma~\ref{lem_listofint}
 and the fact that
 \[
  \matr{0}{-1}{1}{a_i} = ST^{a_i}
 \]
 we find that $A_{i+1} =_{\sltz} A_i ST^{a_i}$ for all
 $i\in\Z$. We conclude that
 \begin{align*}
  A_d =_{\sltz} A_0 ST^{a_0}\ldots ST^{a_{d-1}}
 \end{align*}
 and since $T^c A_0 =_{\sltz} A_d$ this implies that
 \[
  ST^{a_0}\ldots ST^{a_{d-1}} =_{\sltz} A_0^{-1}T^c A_0.
 \]
 Since $S^4$ generates the kernel of the projection $G\to\sltz$ we have that
 \[
  ST^{a_0}\ldots ST^{a_{d-1}} =_G S^{4k} A_0^{-1}T^c A_0
 \]
 for some $k\in\Z$. This is because 
$S^2$ is in the center of $\sltz$ and 
 when reducing an element of $\sltz$ we can assume that
 the relation $S^4 =_{\sltz} I$ is not used until the last step.
 Rearranging we have
 \begin{equation}\label{eqn_lemsthelixeqn}
  A_0 ST^{a_0}\ldots ST^{a_{d-1}} A_0^{-1}T^{-c} =_G S^{4k}.
 \end{equation}
 To complete the proof we must only show that $k=1$ in Equation~\eqref{eqn_lemsthelixeqn}.

 Let $\si,\eta\in G$ be given by
 \[
  \si =_G ST^{a_0}\ldots ST^{a_{d-1}}\,\textrm{ and }\, \eta =_G A_0 \si A_0^{-1} T^{-c}
 \]
 so Equation~\eqref{eqn_lemsthelixeqn} becomes $\eta=_G S^{4k}$.
 Since $W(S^{4k})=k,$ it is sufficient to show that
 $W(\eta) = 1$.  Recall $\pi_1(\sltr)$ is abelean
 so the class of a loop is well-defined without
 fixing the basepoint. By Lemma~\ref{lem_windandW}, $\eta =_{\sltz} I$
 implies that $W(\eta) = \mathrm{wind}(\rho(\eta))$.
 By Lemma~\ref{lem_homotopy} $\rho(\eta)$ is homotopic to
 \[
  \big(\rho'(\eta)\big)(t) = \left\{\begin{array}{ll}
                             \rho(A_0^{-1})(4t)                                     ,& \, 0\leq t \leq \nicefrac{1}{4}\\[2ex]
                             \bigg(\big(\rho(\si)\big)(4t-1)\bigg) A_0^{-1}           ,& \, \nicefrac{1}{4} \leq t \leq \nicefrac{1}{2}\\[2ex]
                             \si A_0^{-1} \bigg(\big(\rho(T^{-c})\big)(4t-2)\bigg)    ,& \, \nicefrac{1}{2} \leq t \leq \nicefrac{3}{4}\\[2ex]
                             \bigg(\big(\rho(A_0)\big)(4t-3)\bigg)\si A_0^{-1} T^{-c} ,& \, \nicefrac{3}{4} \leq t \leq 1.
                         \end{array}
                  \right.
 \]
 Let $\ga_0\colon[0,1]\to\sltr$ be the path from
 $A_0^{-1}$ to itself given by
 \begin{equation}
  \label{eqn_ga0}
  \ga_0(t) = \left\{\begin{array}{ll}
                             \bigg(\big(\rho(\si)\big)(2t)\bigg) A_0^{-1}             ,& \, 0 \leq t \leq \nicefrac{1}{2}\\[2ex]
                             \si A_0^{-1} \bigg(\big(\rho(T^{-c})\big)(2t-1)\bigg)    ,& \, \nicefrac{1}{2} \leq t \leq 1.\\
                         \end{array}
                  \right.
 \end{equation}
 The paths $\ga_0$ and $\rho'(\eta)$ are homotopic via the homotopy
 \[
  \big(\rho'_s(\eta)\big)(t) = \left\{\begin{array}{ll}
                             \rho(A_0^{-1})(\frac{4t}{s})                              ,& \, 0\leq t \leq \nicefrac{s}{4}\\[2ex]
                             \Bigg(\Big(\rho(\si)\Big)\big(\frac{4t-s}{2-s}\big)\Bigg) A_0^{-1}            ,& \, \nicefrac{2}{4} \leq t \leq \nicefrac{1}{2}\\[2ex]
                             \si A_0^{-1} \Bigg(\Big(\rho(T^{-c})\Big)\big(\frac{4t-2}{2-s}\big)\Bigg)     ,& \, \nicefrac{1}{2} \leq t \leq \frac{4-s}{4}\\[2ex]
                             \Bigg(\Big(\rho(A_0)\Big)\big(\frac{4t+s-4}{s}\big)\Bigg)\si A_0^{-1} T^{-c}  ,& \, \frac{4-s}{4} \leq t \leq 1
                         \end{array}
                  \right.
 \]
 for $0<s\leq 1$ where $\ga_0$ is defined as above.
 Thus, to complete the proof we only must show
 $\mathrm{wind}(\ga_0)=1$ where $\ga_0$ is as in
 Equation~\eqref{eqn_ga0}.
 By Lemma~\ref{lem_SThomotopic}, the path
 \[
 \mathrm{pr}\bigg(\rho(\si)(2 (\cdot))\bigg)\colon [0, \nicefrac{1}{2}]\to(\R^2)^*
 \]
 is homotopic to a path which travels counter-clockwise
 at most one full rotation.
 The path
 \[
  \mathrm{pr}\bigg(\si A_0^{-1}\,\rho(T^{-c})(2(\cdot)-1)\bigg)\colon[\nicefrac{1}{2},1]\to(\R^2)^*,
 \]
 travels  only counter-clockwise and cannot cross the
 line $\{y=0\}$, so it completes at most one half-rotation.
 Since $\si A_0^{-1} T^{-c} =_{\sltz} A^{-1}_0$, the path
 $\ga_0$ thus circles the origin an integer number of times, so we
 conclude that $\mathrm{wind}(\mathrm{pr}(\ga_0))= \mathrm{wind}(\ga_0)=1$.
 This completes the proof.
\end{proof}

\begin{figure}
 \centering
 \includegraphics[width = 230pt]{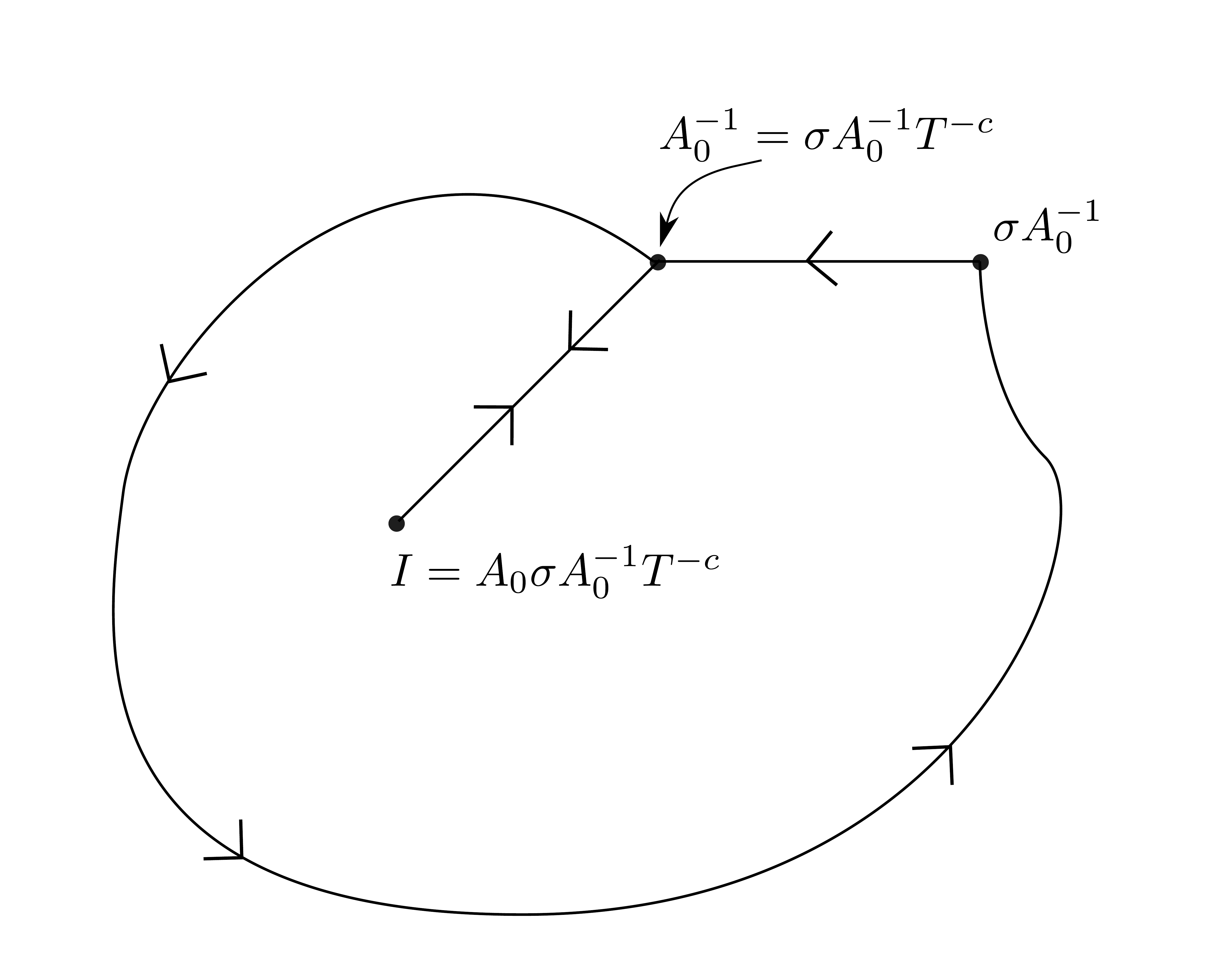}
 \caption{The loop $\rho(\eta)$ from the proof of
  Lemma~\ref{lem_sthelixeqn}, which has winding number 1.}
 \label{fig_sthelixeqnfigure}
\end{figure}

\begin{proof}[Proof of Proposition~\ref{prop_correspondence}]
 Let $\mathcal{H} = (d,c,[\{v_i\}_{i\in\Z}])$ be a semitoric helix.
 Then there exists associated
 integers $(a_0, \ldots, a_{d-1})\in\Z^d$ as in Definition~\ref{def_associatedints}
 by Lemma~\ref{lem_listofint}.
 If $\{w_i\}_{i\in\Z}\in[\{v_i\}_{i\in\Z}]$ then by Definition~\ref{def_sthelix}
 there exists some $k,\ell\in\Z$ such that $v_i = T^k w_{i+\ell}$
 for all $i\in\Z$.  In this case $a_i v_{i+1} = v_i + v_{i+2}$ implies that
 $
  a_i w_{i+1+\ell} = w_{i+\ell} + w_{i+2+\ell}
 $
 and denoting $a_j := a_{j\,\mathrm{mod}d}$ this implies that
 $
  a_{i-\ell} w_{i+1} = w_{i} + w_{i+2}.
 $
 Thus, the associated integers for $\{w_i\}_{i\in\Z}$ are given by
 $
  (a_{-\ell}, a_{1-\ell}, \ldots, a_{d-1-\ell})
 $
 which agrees with those integers for $\{v_i\}_{i\in\Z}$ up to
 cyclic permutation, as desired.

 Suppose $(a_0, \ldots, a_{d-1})\in\Z^d$ is a list of integers satisfying
 \[
  ST^{a_0}\ldots ST^{a_{d-1}} =_G S^4 X^{-1} T^{-c} X
 \]
 for some $c\in\Z_{>0}$.  Let $A_0\in\sltz$ be any matrix
 satisfying $X =_{\sltz} A_0$ and define $v_0, v_1\in\Z^2$
 so that $A_0 = [v_0, v_1]$.  Then define $v_2, \ldots, v_{d-1}$
 by
 $
  v_{i} = a_{i-2} v_{i-1} - v_{i-2}
 $
 for $i=2, \ldots, d-1$.  Use the relationship
 $v_{i+d} = T^c v_i$ to extend $v_0, \ldots, v_{d-1}$
 to $\{v_i\}_{i\in\Z}$.  Since $W(ST^{a_0}\ldots ST^{a_{d-1}})=1$, the vectors
 $v_0, \ldots, v_{d-1}$ are in counter-clockwise order
 and by construction $\det(v_i, v_{i+1})=1$ for all
 $i\in\Z$, so $[\{v_i\}_{i\in\Z}]$ is a semitoric helix
 with the prescribed associated integers.

 If $\mathcal{H}$ and $\mathcal{H}'$ satisfy $\mathcal{H} = \pm\mathcal{H}'$
 then they have the same associated integers since those integers are defined by a linear
 equation, Equation~\eqref{eqn_virecurrence}, which is invariant under the action of $-I$.

 Conversely, suppose that $\mathcal{H}=(d,c,[\{v_i\}_{i\in\Z}])$
 and $\mathcal{H}'=(d,c,[\{v_i'\}_{i\in\Z}])$ are semitoric
 helices of the same length, complexity, and associated integers.
 Let $A_0 = [v_0,v_1]$ and $A_0' = [v_0',v_1']$ and let
 $a_0, \ldots, a_{d-1}$ be the common associated integers.  Then
 \[
  ST^{a_0}\ldots ST^{a_{d-1}} =_{\sltz} A_0^{-1}T^c A_0
 \,\textrm{ and }\,
  ST^{a_0}\ldots ST^{a_{d-1}} =_{\sltz} (A_0')^{-1}T^c (A_0')
 \]
so $ A_0^{-1}T^c A_0 =_{\sltz} (A_0')^{-1}T^c (A_0')$.
Thus $A_0' A_0^{-1}$ commutes with $T^c$, so
$A_0' A_0^{-1} =_{\sltz} \pm T^k$ for some $k\in\Z$, and
so we may assume $A_0 =_{\sltz}\pm A_0'$ because $T^k$ is already included in
the equivalence relation on helices.
Finally, $v_0 = \pm v_0'$ and $v_1 = \pm v_1'$ implies
$v_i = \pm v_i'$ by Equation~\eqref{eqn_virecurrence}.
\end{proof}

\section{Standard form in \texorpdfstring{$\mathrm{PSL}_2(\Z)$}{the projective special linear group} and the winding number}
\label{sec_pltz}

First we prove several lemmas which will be needed
in the proof of Theorem~\ref{thm_standardpsltz}.

\begin{lemma}\label{lem_nonnegwinding}
 If $\sigma\in\STgroup$ is $S$\--positive
 and $\sigma=_{\psltz} I$ then $W(\sigma)\geq 0$
 where $W(\sigma)=0$ if and only if $\sigma$ is the empty word.
\end{lemma}

\begin{proof}
If $\sigma$ is the empty word then $W(\sigma) = 0$ and the claim holds.
Assume $\sigma$ is not the empty word.
Since $\sigma$ is $S$\--positive up to conjugation by $T$, which does not
change $W(\si)$, we may write it as
$
 \sigma =_{\STgroup}ST^{a_0}\ldots ST^{a_d-1}
$
for some $a_0, \ldots, a_{d-1}\in\Z$. We define a sequence of vectors
$
 v_0, \ldots, v_{d-1}\in\Z^2
$
by choosing any $v_0,v_1\in\Z^2$ with $\det(v_0,v_1)=1$
and defining $v_2, \ldots, v_{d-1}$ by
$
 v_{i+2} = -v_i + a_i v_{i+1}
$
for $i=0, \ldots, d-3$.
Let $\ga\colon[0,1]\to(\R^2)^*$ be a path which connects $v_0, \ldots, v_{d-1}$
in order and travels only counter-clockwise.  Then, $W(\si) = \mathrm{wind}(\ga)$
and $\mathrm{wind}(\ga)>0$ because $\ga$ must travel at least once around the origin
to move only counter-clockwise and return to $\ga(0)$.
\end{proof}

\begin{lemma}\label{lem_minimalwinding}
 If $X\in\psltz$ then there exists some $q\in\frac{1}{12}\Z$ such that
 $w(\sigma)\geq q$ for all $\sigma\in\STgroup$
 which are $S$\--positive and satisfy
 $\sigma =_{\psltz} X$.
\end{lemma}

\begin{proof}
 Since $S =_{\psltz} S^{-1}$ every element of $\psltz$ has an 
 $S$\--positive representation.
 Fix some $S$\--positive $\eta\in\STgroup$ such that 
 $\eta=_{\psltz}X^{-1}$
 and let $q=-W(\eta)$.
 Let $\sigma$ be any $S$\--positive element of $\STgroup$ such that
 $\sigma = _{\psltz} X$.  Now $\sigma \eta =_{\sltz} I$, so $W(\sigma \eta)\geq 0$
 by Lemma~\ref{lem_nonnegwinding}.  This means
 $W(\sigma)+W(\eta)\geq 0 $ so $W(\sigma) \geq q$ and the result follows because
 $q$ does not depend on the choice of $\sigma$.
\end{proof}

The following is a special case of~\cite[Lemma 3.8]{KaPaPe2016}, but for the sake
of being self-contained we include the proof here.
\begin{lemma}\label{lem_ai}
 Suppose $d>0$ and $b,a_0, \ldots, a_{d-1}\in\Z$ are such that
 \begin{equation}\label{eqn_ai}
  T^b ST^{a_0}\ldots ST^{a_{d-1}} =_{\psltz} I.
 \end{equation}
 Then $d>1$. If $d=2$ then $a_0 = 0$ and $a_1 = -b$.
 If $d>2$ then $a_i \in \{0, \pm 1\}$ for some $0\leq i<d-2$.
\end{lemma}

\begin{proof}
 The group $\psltz$ acts faithfully on the extended real
 line $\R\cup\{\infty\}$ by
 $
  T(x) = x+1 ,\,\, S(x) = \frac{-1}{x}
 $
 for $x\in\R\setminus\{0\}$, $T(0)=0$, $S(0)=\infty$, $T(\infty)=\infty$,
 and $S(\infty)=0$.

 If $d=1$, then Equation~\eqref{eqn_ai} states that
 $
  T^b ST^{a_0} =_{\psltz} I
 $
 which is impossible because
 $
  T^{b}ST^{a_0}(\infty) = b \neq \infty.
 $

 If $d=2$, then, after conjugating by $ST^{a_1}$, Equation~\eqref{eqn_ai}
 states that
 $
  ST^{a_1+b}ST^{a_0} =_{\psltz} I
 $
 and evaluation of both sides at infinity gives
 $
  S(a_1+b) = \infty
 $
 which implies $a_1 = -b$. Thus, $S^2T^{a_0} =_{\psltz} I$ so $a_0 = 0$.

 Finally, suppose $d>2$ and $a_i\notin \{0,\pm 1\}$ for all $i = 0, \ldots, d-2$.
 Conjugate Equation~\eqref{eqn_ai} by $ST^{a_{d-1}}$ to produce
 \begin{equation}\label{eqn_psltzeqn}
  ST^{a_{d-1}+b}ST^{a_0}\ldots ST^{a_{d-2}} =_{\psltz} I.
 \end{equation}
 Let $y = ST^{a_0}\ldots ST^{a_{d-2}}(\infty)$ so Equation~\eqref{eqn_psltzeqn}
 implies $ST^{a_{d-1}+b}(y)=\infty$. On the other hand,
 $ST^{a_{d-2}}(\infty)=0$ and
 if $x\in\R$ with $\abs{x}<1$ then $0<\abs{ST^k(x)}<1$ for any integer $k\notin\{\pm1, 0\}$, so
 \[
  \abs{y} = \abs{ST^{a_0}\ldots ST^{a_{d-3}}(0)}\in(0,1),
 \]
 and thus
 \[
  ST^{a_{d-1}+b}(y) = \frac{1}{y +a_{d-1}+b} \neq \infty,
 \]
 forming a contradiction.
\end{proof}

\begin{lemma}\label{lem_STinvSequalsTST}
 $ST^{-n}S =_{\psltz} (TST)^{n}$ for $n\geq0$.
\end{lemma}
\begin{proof}
 First $STS=_{\psltz}T^{-1}ST^{-1}$ implies $S=_{\psltz}TSTST$ so
 $ST^{-1}S =_{\psltz} TST$
 since $S=_{\psltz}S^{-1}$.  Now,
 \[
  ST^{-n}S =_{\psltz} (ST^{-1}S)^n =_{\psltz} (TST)^n
 \]
 for $n>0$, and if $n=0$ the claim reduces
 to $S^2 =_{\psltz} I$.
\end{proof}

\subsection{Standard form for elements of \tex{$\psltz$}}

In this section we prove Theorem~\ref{thm_standardpsltz}.

\begin{proof}[Proof of Theorem~\ref{thm_standardpsltz}]
 Let $\sigma\in\STgroup$ any $S$\--positive word with
 $\sigma =_{\psltz} X$.
 There are three steps to the reduction algorithm we will use on $\sigma$,
 where~\ref{reduction2} holds by Lemma~\ref{lem_STinvSequalsTST}.
 The reductions are:
 \begin{enumerate}[noitemsep, label = {\bf Reduction \arabic*}, align = left]
  \item\label{reduction1} replace $S^2$ with $I$;
  \item\label{reduction2} replace $ST^{-n}S$ with $(TST)^n$, for some $n>0$;
  \item\label{reduction3} replace $STS$ with $T^{-1}ST^{-1}$;
 \end{enumerate}
 To reduce the word we iteratively apply~\ref{reduction1},
 \ref{reduction2}, and \ref{reduction3} until no more are
 possible.
 Each of these reductions preserves the value of
 $\sigma$ in $\psltz$ and recall that the winding number cannot decrease
 indefinitely by Lemma~\ref{lem_minimalwinding}.  \ref{reduction1}
 and~\ref{reduction2} reduce the winding number 
 while~\ref{reduction3}
 preserves the winding number but reduces the number of times $S$ appears in the word,
 which is bounded below by zero. Thus, this process must terminate and
 after the reduction the word will be of the required form.

 Now we will show uniqueness.  Suppose that
 $
  \si,\eta\in\STgroup
 $
 with
 $\si =_{\psltz} \eta$ and
$
  \si =_{\STgroup} T^b ST^{a_0}\ldots ST^{a_{d-1}}$,
  $\eta =_{\STgroup} T^{b'} ST^{a'_0}\ldots ST^{a'_{d'-1}}$,
 where $a_i, a'_j>1$ for $i=0, \ldots, d-2$ and $j=0, \ldots d'-2$.
 First assume $\mathrm{min}(d,d')\leq 1$, and in this case assume $d\geq d'$.

 If $d'=0$ then
 $
  T^{b-b'}ST^{a_0}\ldots ST^{a_{d-1}} =_{\psltz} I
 $
 which contradicts Lemma~\ref{lem_ai} unless $d=0$, in which case
 $T^{b-b'}=_{\psltz} I$ so $b=b'$.
 If $d' = 1$ then
 \[
  T^{b-b'}ST^{a_0}\ldots ST^{a_{d-1}-a'_0}ST^0 =_{\psltz} I
 \]
 so $a_{d-1}-a_0'\in\{0,\pm1\}$ by Lemma~\ref{lem_ai}.
 Consider the cases if $d>1$.  If $a_{d-1}-a_0' = 0$ then
 \[
  T^{b-b'}ST^{a_0}\ldots ST^{a_{d-2}}S^2 =_{\psltz} I
 \]
 which contradicts Lemma~\ref{lem_ai} after replacing $S^2$
 by $I$.  If $a_{d-1}-a_0' = -1$ then
 \[
  T^{b-b'}ST^{a_0}\ldots ST^{a_{d-2}}ST^{-1}S =_{\psltz} I
 \]
 which contradicts Lemma~\ref{lem_ai} after replacing
 $ST^{-1}S$ by $TST$.  Finally, if $a_{d-1}-a_0' = 1$, then
 \[
  T^{b-b'}ST^{a_0}\ldots ST^{a_{d-2}-1}ST^{-1} =_{\psltz} I
 \]
 which contradicts Lemma~\ref{lem_ai} unless $a_{d-2}=2$.
 This process is repeated to conclude that
 $a_0 = \ldots = a_{d-2}=2$ so
 $
  T^{b-b'}(ST^2)^{d-1}STS =_{\psltz} I
 $
 which implies
 $
  T^{b-b'-1}ST^{-d}=_{\psltz} I.
 $
 By Lemma~\ref{lem_ai} this cannot hold.  Thus, $d=1$, in which
 case
 $
  T^{b-b'}ST^{a_0-a_0'}S =_{\psltz} I,
 $
 so $b-b' = 0$ and $a_0 - a_0' = 0$ by Lemma~\ref{lem_ai}.

 Finally, assume $d,d'>1$ and assume that $a_{d-1}\neq a'_{d'-1}$, 
 otherwise cancel
 $ST^{a_{d-1}}$ from both sides.
 In this case we see that
 $\si\eta^{-1}=_{\psltz}I$
 implies
 \[
  T^b ST^{a_0}\ldots ST^{a_{d-1}-a_{d'-1}'}ST^{-a_{d'-2}'}\ldots ST^{-a_0'}S =_{\psltz} I
 \]
 and since some power of $T$ must be in $\{0,\pm 1\}$ by Lemma~\ref{lem_ai},
 but $a_i,a_j'>1$ for $i=0, \ldots, d-2$, $j= 0, \ldots, d'-2$ since $\si$
 and $\eta$ are in standard form, we conclude $a_{d-1}-a'_{d'-1}\in\{0,\pm 1\}$.
 We have assumed $a_{d-1}\neq a'_{d'-1}$
 so ${a_{d-1}-a'_{d'-1}}=\pm 1$, and furthermore we can assume
 ${a_{d-1}-a'_{d'-1}}= 1$, otherwise exchange $\si$ and $\eta$. Then
 choose maximal $k\in\Z_{\geq 0}$ such that
 $a_{d-2} = a_{d-3} = \ldots = a_{d-2-(k-1)} = 2$.
 where $k=0$ if $a_{d-2}\neq 2$.
 If $k<d-1$ then
 \begin{align*}
  \si\eta^{-1} &=_{\psltz} T^{b}ST^{a_0}\ldots ST^{a_{d-2-k}}(ST^2)^k(STS)T^{-a'_{d'-2}}\ldots ST^{-a'_0}ST^{-b'}\\
          &=_{\psltz} T^{b}ST^{a_0}\ldots ST^{a_{d-2-k}-1}(TST)^kST^{-a'_{d'-2}-1}\ldots ST^{-a'_0}ST^{-b'}\\
          &=_{\psltz} T^{b}ST^{a_0}\ldots ST^{a_{d-2-k}-1}ST^{-a'_{d'-2}-k-1}\ldots ST^{-a'_0}ST^{-b'}.
 \end{align*}
  Since $a_{d-2-k}-1>1$ and $-a'_{d'-2}-k-1<-1$ this expression cannot evaluate
  to the identity in $\psltz$ by Lemma~\ref{lem_ai}.

  Otherwise, $k = d-1$, in which case
 \begin{align*}
  \si\eta^{-1} &=_{\psltz} T^{b}(ST^2)^{d-1}(STS)T^{-a'_{d'-2}}\ldots ST^{-a'_0}ST^{-b'}\\
          &=_{\psltz} T^{b-1}(TST)^{d-1}ST^{-a'_{d'-2}-1}\ldots ST^{-a'_0}ST^{-b'}\\
          &=_{\psltz} T^{b-1}ST^{-a'_{d'-2}-d}\ldots ST^{-a'_0}ST^{-b'}.
 \end{align*}
 which again cannot evaluate to the identity in $\psltz$ by Lemma~\ref{lem_ai}.
 This completes the proof of uniqueness.

 Lastly, we will show the standard form has minimal winding number.
 Let $X\in\psltz$ and suppose $\eta\in\STgroup$ is $S$\--positive
 with $\eta =_{\psltz} X$.  Then $\eta$ can be reduced to
 the standard form of $X$, denoted $\overline{X}\in\STgroup$, by following the reduction algorithm
 at the beginning of the proof. Since each 
 of~\ref{reduction1}\--\ref{reduction3} in the algorithm either preserves
 or reduces the winding number, $W(\overline{X})\leq W(\eta)$.
\end{proof}

\subsection{Standard forms and the winding number}

Recall that given any $X\in\psltz$ we denote by
$
 \overline{X}\in\STgroup
$
the standard form of $X$, as given in Theorem~\ref{thm_standardpsltz}.

\begin{lemma}\label{lem_wxpluswxbar}
 If $X\in\psltz\setminus \{T^k\}_{k\in\Z}$ then
 \[
  W\big(\overline{X}\big) + W\big(\overline{X^{-1}}\big) = \frac{1}{2}.
 \]
\end{lemma}

\begin{proof}
 Write
 $
  \overline{X} = T^b ST^{a_0}\ldots ST^{a_{d-1}}
 $
 and since $X\neq T^k$ for any $k\in\Z$, $d>0$.
 Now, $W\big(\big(\overline{X}\big)^{-1}\big) = -W\big(\overline{X}\big)$ where
 \[
  \big(\overline{X}\big)^{-1} = S^{-1}T^{-a_{d-1}}\ldots S^{-1}T^{-a_0}S^{-1}T^{-b}.
 \]
 We will reduce $\big(\overline{X}\big)^{-1}$ to standard form 
 using the reduction steps in the proof of Theorem~\ref{thm_standardpsltz} 
 and keep track
 of the winding number.  Replacing each $S^{-1}$ by $S$ increases the winding number
 by $\nicefrac{d}{2}$.  Now replace each $ST^{-a_{i}}S$ with $(TST)^{a_i}$ for each even index $i$
 which at most increases the odd indexed powers of $T$ by $2$.
 Since each $a_i\geq 2$ for $i=0, \ldots, d-2$ we do the replacement
 $ST^{-a_i+2}S = (TST)^{a_i-2}$ for odd $0<i<d-3$ and the replacement
 $ST^{-a_i+1}S = (TST)^{a_i-1}$ for $i=1$ and the highest odd $i\leq d-2$.
 Thus we have now used $ST^{-n}S = (TST)^n$, for varying values of $n>0$,
 a total of $d-1$ times decreasing $W$ by $\nicefrac{1}{2}$ each time.  The word produced in this way
 is now in standard form so it is equal to $\overline{X^{-1}}$ and
 \[
  W\big(\overline{X^{-1}}\big) = -W\big(\overline{X}\big)+\frac{d}{2} -\frac{d-1}{2} = -W\big(\overline{X}\big)+\frac{1}{2}
 \]
 as desired.
\end{proof}

We can now prove that in many cases the first power of $T$ in $\overline{X}$ and
the last power of $T$ in $\overline{X^{-1}}$ must sum to $1$.

\begin{lemma}\label{lem_wraparoundinverses}
 For $X\in\psltz$ write
 \[
  \overline{X} =_{\STgroup} T^bST^{a_0}\ldots ST^{a_{d-1}}
  \textrm{ and }
  \overline{X^{-1}} =_{\STgroup} T^{b'}ST^{a_0'}\ldots ST^{a_{d'-1}'}.
 \]
 Then
 \[a_{d-1}+b' = a_{d'-1}'+b=0\] if $X=_{\psltz}T^k S T^a$  or $X =_{\psltz} T^k$  for some $k,a\in\Z$,
 and
 \[
  a_{d-1}+b' = a_{d'-1}'+b=1
 \]
  otherwise.
\end{lemma}

\begin{proof}
 The cases of $X=_{\psltz}T^kST^a$ and $X=_{\psltz}T^k$ are easily checked.
 Suppose $X$ is not of that form.
 Since $\overline{X^{-1}}\,\overline{X}=_{\sltz}I$
 by Lemma~\ref{lem_ai} some power of $T$ that is not at the front or end of the word
 must be $-1$, $1$, or $0$. Since $\overline{X}$ and $\overline{X^{-1}}$ are in
 standard form, $X\neq_{\psltz} T^k S T^a$, and $X\neq_{\psltz} T^k$ this means that $a_{d'-1}'+b\in\{\pm 1 , 0\}$.

 If $a_{d'-1}'+b=0$ then $S^2$ is a subword of $\overline{X^{-1}}\,\overline{X}$
 which can be replaced by $I$ and if $a_{d'-1}'+b = -1$ then $ST^{-1}S$
 is a subword of $\overline{X^{-1}}\,\overline{X}$ which can be replaced by
 $TST$. In either case this means that
 $
  W\big(\overline{X^{-1}X}\big)\leq W\big(\overline{X^{-1}}\big) + W\big(\overline{X}\big)-\frac{1}{2} = 0
 $
 where the last equality is by Lemma~\ref{lem_wxpluswxbar}.
 By Lemma~\ref{lem_nonnegwinding} $W(\overline{X^{-1}X})\geq 0$
 with equality only when $X=I$.
 Since $X\neq I$ we must have $a_{d'-1}'+b=1$.
 The same analysis on $\overline{X}\,\overline{X^{-1}}$ implies
 that $a_{d-1} + b' = 1$.
\end{proof}

\begin{lemma}\label{lem_XTXminimal}
 Let $X\in\psltz$ and $c\in\Z_{>0}$.
 Then 
 $\overline{X^{-1}T^cX} =_G \overline{X^{-1}}T^c\overline{X}$ and
 in particular
 $W(\overline{X^{-1}T^cX}) = W(\overline{X^{-1}}T^c\overline{X}).$
\end{lemma}
\begin{proof}
If $X=_{\psltz}T^k$ for some $k\in Z$ then
$
 \overline{X^{-1}T^cX} =_{\STgroup} \overline{X^{-1}}T^c\overline{X} =_{\STgroup} T^c
$
so the result holds.  If $X=_{\psltz} T^kST^a$ for some $k,a\in\Z$ then there are two cases. If $c>1$ then
$
 \overline{X^{-1}T^cX} =_{\STgroup} \overline{X^{-1}}T^c\overline{X} =_{\STgroup} T^kST^cST^{-k}
$
so the result holds.
If $c=1$ then
$
 \overline{X^{-1}T^cX} =_{\STgroup} T^{k-1}ST^{-k-1}
$
while
$
 \overline{X^{-1}}T^c\overline{X} =_{\STgroup} T^kSTST^{-k}
$
and the result still holds.

If $X\neq_{\psltz} T^k$ and $X\neq_{\psltz} T^k S T^a$ for all $k,a\in\Z$, then $\overline{X^{-1}}T^c\overline{X}$ is already in standard form
for any $c>0$ by Lemma~\ref{lem_wraparoundinverses}, so
$\overline{X^{-1}T^cX} =_{\STgroup} \overline{X^{-1}}T^c\overline{X}.$
\end{proof}

\section{Minimal models for semitoric helices}
\label{sec_minhelices}

\begin{definition}
 An $S$\--positive word with no leading $T$,
 $ST^{a_0}\ldots ST^{a_{d-1}}\in\STgroup$,
is \emph{minimal} if and only if $a_0, \ldots, a_{d-1}\neq 1$.
\end{definition}
Minimal words are those associated to minimal helices.

\begin{lemma}\label{lem_reducingminwords}
 Suppose $\sigma=ST^{a_0}\ldots ST^{a_{d-1}}\in\STgroup$ is minimal
 and there exists $X\in G\setminus\{S^{2\ell}T^k\}_{\ell, k\in\Z}$ such that
 \[
  \sigma =_G S^4 X^{-1}T^c X.
 \]
 Then, after cyclically reordering $a_0, \ldots, a_{d-1}$ if necessary, $a_0\leq 0$
 and one of the following hold:
\begin{enumerate}[label=\textup{(\roman*)}]
 \item $a_0 = 0$ and $\overline{\sigma} =_{\STgroup} T^{a_1}ST^{a_2}\ldots ST^{a_{d-1}}$;
 \item $a_0 < 0$ and $\overline{\sigma} =_{\STgroup} (TST)^{-a_0}T^{a_1}ST^{a_2}\ldots ST^{a_{d-1}}$.
\end{enumerate}
\end{lemma}

\begin{proof}
 Notice
 $
  W(\sigma) = W(S^4 X^{-1}T^c X) = 1 -\frac{c}{12}
 $
 while
 \begin{align*}
  W(\overline{\sigma}) &= W(\overline{X^{-1}T^cX}) = W(\overline{X^{-1}}T^c\overline{X}) \\&= W(\overline{X^{-1}}) + W(\overline{X}) - \frac{c}{12} = \frac{1}{2}-\frac{c}{12}
 \end{align*}
 by Lemmas~\ref{lem_XTXminimal} and~\ref{lem_wxpluswxbar} since $X\neq_{\psltz} T^k$ for any $k\in\Z$.
 Thus, $W(\sigma) \neq  W(\overline{\sigma})$ so $\sigma$ is not in
 standard form.  This means that $a_j \leq 1$ for some fixed $j\in\{0, \ldots, d-2\}$
 and since $\sigma$ is minimal this implies $a_j\leq 0$.

 If $a_j = 0$ for some $j\in\{0, \ldots, d-2\}$ then reorder so $a_0=0$ and
 $\sigma = S^2T^{a_1}ST^{a_2}\ldots ST^{a_{d-1}}$.
 Notice that
 $
  \eta = T^{a_1}ST^{a_2}\ldots ST^{a_{d-1}}
 $
 satisfies $\eta =_{\psltz} \si$ so 
 $\overline{\eta} =_{\STgroup}\overline{\si}$ 
 and also notice $W(\eta) = W(\si) - \nicefrac{1}{2} = W(\overline{\sigma})$.  
 All steps
 in the reduction algorithm in the proof of Theorem~\ref{thm_standardpsltz}
 reduce the winding number, except for the blowdown 
 $STS\rightarrow T^{-1}ST^{-1}$,
 so the only possible step to reduce $\eta$ into standard form is a blowdown.
 For a blowdown to be possible we must have $a_j =1$ for some
 $j\in\{1, \ldots, d-1\}$, contradicting the minimality
 of $\sigma$.  Thus, $\eta =_{\STgroup}\overline{\eta}$ so
 $\overline{\sigma} =_{\STgroup} \eta$.

 Otherwise, $a_j \neq 0$ for all $j\in\{0, \ldots, d-1\}$ so, after 
 cyclically reordering, we may assume
 $a_0 <0$.  In this case let
 \[
  \eta' = (TST)^{-a_0}T^{a_1}ST^{a_2}\ldots ST^{a_1}
 \]
 and notice $\eta' =_{\psltz} \sigma$ so 
 $\overline{\eta'} =_{\STgroup}\overline{\si}$.
 Again, $W(\eta') = W(\overline{\sigma})$ so the only
 possible reduction move would be a blowdown, but
 if a blowdown could be performed on $\eta'$ that
 would contradict the minimality of $\sigma$, except
 in the case that $a_1=0$, which we have assumed does not occur.
 Thus, $\overline{\sigma} = \eta'$.
\end{proof}

Here we classify all words associated to minimal
semitoric helices.  Recall $\mathcal{S}$ from
Equation~\eqref{eqn_mathcalS}.

\begin{lemma}[Classification of minimal words]
 \label{lem_classify}
 Let $\mathcal{H}$ be a
 helix with complexity $c>0$.
 If $\mathcal{H}$ is minimal then there is an associated
 word $\si\in\STgroup$ which is exactly
 one of the following, where
 $A_0 = [v_0, v_1]$ for some
 $\{v_i\}_{i\in\Z}$ such that
 $\mathcal{H} =(d,c,[\{v_i\}_{i\in\Z}])$.
 \vspace{3pt}
 \newcounter{typecountword}
 \setcounter{typecountword}{0}
 \begin{center}
 \begingroup
 \setlength{\tabcolsep}{3.5pt}
 \begin{tabular}{l|l|c|rcl}
      \rm{type} &   \multicolumn{1}{|c|}{$\si\in\STgroup$}     & $c$ & &$A_0$& \\[0.5ex]
      \hline
  \refstepcounter{typecountword}\rm{(\thetypecountword)}\label{type1word} & $\si= ST^{-1}ST^{-4}$                           & $c=1$     & $ ST^{-2}   $&$=$&$ \matr{0}{-1}{1}{-2}$\rule{0pt}{4.5ex}  \\[3ex]
  \refstepcounter{typecountword}\rm{(\thetypecountword)}\label{type2word} & $\si = ST^{-2}ST^{-2}$                          & $c=2$     & $ ST^{-1}   $&$=$&$ \matr{0}{-1}{1}{-1}$  \\[3ex]
  \refstepcounter{typecountword}\rm{(\thetypecountword)}\label{type3word} & $\si= S^2T^aST^{-a-2}$         , $a\neq1, -3$   & $c=1$     & $ ST^{-a-1} $&$=$&$ \matr{0}{-1}{1}{-a-1}$  \\[3ex]
  \refstepcounter{typecountword}\rm{(\thetypecountword)}\label{type4word} & $\si= ST^{-1}ST^{-1}ST^{c-1}$                   & $c\neq2$  & $ I         $&$=$&$ \matr{1}{0}{0}{1}$  \\[3ex]
  \refstepcounter{typecountword}\rm{(\thetypecountword)}\label{type5word} & $\si= S^2T^aST^cST^{-a}$       , $a\neq\pm 1$   & $c\neq 1$ & $ ST^{-a}   $&$=$&$ \matr{0}{-1}{1}{-a}$  \\[3ex]
  \refstepcounter{typecountword}\rm{(\thetypecountword)}\label{type6word} & $\si= S^2T^aS^2T^{c-a}$        , $a\neq1, c-1$  & $c>0$     & $ I         $&$=$&$ \matr{1}{0}{0}{1}$  \\[3ex] 
  \refstepcounter{typecountword}\rm{(\thetypecountword)}\label{type7word} & $\si= S^2\overline{A_0^{-1}}T^c\overline{A_0}$  & $c>0$ & \,\,$A_0$&$\in$&$\mathcal{S}$ 
 \end{tabular}
 \endgroup
 \end{center}
 \vspace{3pt}
 where $a\in\Z$ is a parameter.
\end{lemma}

\begin{proof}
 Suppose that $\si=_{\STgroup} ST^{a_0}\ldots ST^{a_{d-1}}$
 is minimal and associated to a semitoric helix
 $\mathcal{H}$ of length $d$ and complexity $c>0$.
 By Lemma~\ref{lem_sthelixeqn}
 there exists some $X\in G$ such that
 \begin{equation}\label{eqn_sthelixproofG}
  \si =_G S^4 X^{-1}T^c X.
 \end{equation}
 We will proceed by cases on $X$, and show that in each case
 that $\si$ is one of type~\eqref{type1word}-\eqref{type7word}
 in the statement of the Lemma.

 \emph{Case I: $X=_{\psltz} T^k$ for some $k\in\Z$.}  This implies that
 \[
  ST^{a_0}\ldots ST^{a_{d-1}-c} =_G S^4,
 \]
 and so $(a_0, \ldots, a_{d-2}, a_{d-1}-c)\in\Z^d$ are associated
 to a minimal toric fan.  Such words are completely classified in
 \cite[Lemma 4.8]{KaPaPe2016} and we conclude either
 $d=3$ and $a_0 = a_1 = a_2-c=-1$, which is minimal
 only when $c\neq 2$, or $d=4$ and
 $a_0 = a_2 =0$, $a_3 = c- a_1$, which is minimal
 only when $a\neq 1, c-1$.  Thus $\si$ is either of type~\eqref{type4word}
 or~\eqref{type6word}.

 \emph{Case II: $X\neq_{\psltz} T^k$ for all $k\in\Z$.}
 In light of Equation~\eqref{eqn_sthelixproofG} apply Lemma~\ref{lem_reducingminwords}
 to $\si$ and conclude that,
 after passing to an equivalent helix by cyclically permuting,
 \[
  \sigma =_{\STgroup} ST^{a_0}\ldots ST^{a_{d-1}}
 \]
 satisfies either
 \begin{enumerate}[nosep]
  \item $a_0 = 0$; or
  \item $a_j\neq0$ for all $j=0, \ldots, d-1$ and $a_0 <0$.
 \end{enumerate}
 If $a_0=0$ then
 \[
  \overline{\sigma} = T^{a_1}ST^{a_2}\ldots ST^{a_{d-1}}
 \]
 and otherwise
 \[
  \overline{\sigma} = (TST)^{-a_0}T^{a_1}ST^{a_2}\ldots ST^{a_{d-1}}.
 \]
 By cyclically permuting the $a_i$, which corresponds
 to conjugating Equation~\eqref{eqn_sthelixproofG} by $ST^{a_i}$ for various
 $i$, we change $X$, but $X\neq T^k$ still holds.
 We now have three further cases on $X$.

 \emph{Case IIa: $X=_{\psltz}T^k S T^a$ for $k,a\in\Z$.}
 First assume $a_0=0$.
 If $c=1$ then
 \[
  \overline{X^{-1}T^cX} =_{\STgroup} T^{-a-1}ST^{a-1}
 \]
 so
 $
  \si = S^2 T^{-a-1}ST^{a-1}
 $
 which is minimal
 for $a\neq \pm 2$ and is of type~\eqref{type3word}.
 If $c\neq 1$ then
 \[
  \overline{X^{-1}T^cX} = T^{-a}ST^cST^a
 \]
 so $\si = S^2T^{-a}ST^cST^a$ which is minimal if $a\neq \pm 1$
 and is of type~\eqref{type5word}.

 Now suppose $a_0\neq 0$.  Then
 \[
  \overline{\si} = (TST)^{-a_0}T^{a_1}ST^{a_2}\ldots ST^{d-1}
 \]
 and $a_0<0$.
 If $c=1$ then
 $\overline{\si} = T^{-a-1}ST^{a-1}$
 so $a=-2$ and thus $\overline{\si} = (TST)T^{-4}$
 so $\si = ST^{-1}ST^{-4}$, which is of type \eqref{type1word}.
 If $c=2$, then $\overline{\si} = T^{-a}ST^2ST^a$ which
 means $a=-1$ and $a_0 = -2$ so $\si = ST^{-2}ST^{-2}$,
 which is of type~\eqref{type2word}.
 If $c>2$, then $\overline{\si} = T^{-a}ST^cST^a$ which
 means $a=-1$ and $a_0 = -1$ so $\si = ST^{-1}ST^{c-1}ST^{-1}$,
 which is of type~\eqref{type4word}.

 \emph{Case IIb: $X\neq_{\psltz}T^k, T^kST^a$ for all $k,a\in\Z$ and $a_i\neq 0$ for all $i=0,\ldots,d-1$.}
 In this case $a_0 < 0$.
 If $d=2$, then $\si = ST^{a_0}ST^{a_1}$ and
 $\overline{\si} = (TST)^{-a_0}T^{a_1}$ which means
 $(TST)^{-a_0}T^{a_1} =_{\STgroup} \overline{X^{-1}}T^c\overline{X}$.
 Since $\overline{X^{-1}}T^c\overline{X}$ starts
 with $TS$ it must end with $S$ by Lemma~\ref{lem_wraparoundinverses}
 so $a_1 = -1$.  Now $W(\si) = 1-\nicefrac{c}{12}$
 from Equation~\eqref{eqn_sthelixproofG} and 
 \[
  W(\si) = W(ST^{-n}ST^{-1}) = \frac{1}{12}(7+n)
 \]
 so $n = 5-c$
 and we have
 \begin{equation}\label{eqn_mainthmproof}
  (TST)^{5-c} T^{-1} =_{\STgroup} \overline{X^{-1}}T^c\overline{X}.
 \end{equation}
 The right side of Equation~\eqref{eqn_mainthmproof}
 contains $T^{c+1}$ while the highest power of $T$ on the
 left side is $T^2$, so $c>0$ implies $c = 1$. Thus we obtain
 $\si = ST^{-4}ST^{-1}$ and have type~\eqref{type3word}.

 If $d>2$ then
 \[
  (TST)^{-a_0} T^{a_1}ST^{a_2}\ldots ST^{a_{d-1}} =_{\STgroup} \overline{X^{-1}}T^c \overline{X}
 \]
 implies that $\overline{X^{-1}}T^c \overline{X}$
 must end with $S$ by
 Lemma~\ref{lem_wraparoundinverses}, so $a_{d-1}=0$ which
 contradicts our assumption in this case.

 \emph{Case IIc: $X\neq_{\psltz}T^k, T^kST^a$ for all $k,a\in\Z$ and $a_i=0$ for some $i\in\{0, \ldots, d-1\}$.}
 We have already used Lemma~\ref{lem_reducingminwords} to establish
 that either $a_0=0$ or $a_j\neq 0$ for all $j=0,\ldots, d-1$.
 Thus, in this case, $a_0=0$ so
 \[
  \si = S^2 \overline{\si} = S^2 \overline{X^{-1}}T^c\overline{X}
 \]
 which is minimal if $\overline{X}$ does not end with $ST$
 and $\overline{X^{-1}}$ does not begin with $TS$, and is
 of type~\eqref{type7word}.
\end{proof}

\begin{proof}[Proof of Theorem~\ref{thm_classifyvectors}]
 Let $\mathcal{H} = (d,c,[\{v_i\}_{i\in\Z}])$ be a minimal semitoric helix
 with associated integers
 $(a_0, \ldots, a_{d-1})\in\Z^d$.  Then
 $\sigma = ST^{a_0}\ldots ST^{a_{d-1}}$
 is a minimal word and, passing to an equivalent helix
 if necessary, we conclude $\sigma$ must be of some
 type~\eqref{type1word}-\eqref{type7word} in Lemma~\ref{lem_classify}.
 Types~\eqref{type1word}-\eqref{type6word} for $\sigma$ in Lemma~\ref{lem_classify}
 correspond to types~\eqref{type1}-\eqref{type6} for $\mathcal{H}$ in
 Theorem~\ref{thm_classifyvectors}.  Notice these
 each have length $d<5$.

 Otherwise, $\sigma$ must be of type~\eqref{type7word},
 which means there exists some $X =_G A_0$, where
 $A_0 = [v_0,v_1]\in\mathcal{S}$ and
 \[
  \sigma =_{\STgroup} S^2 \overline{X^{-1}}T^c\overline{X}.
 \]
 Since $A_0 \in\mathcal{S}$ notice that
 $A_0 = ST^{a_0}\ldots ST^{a_{\ell-1}}$ with $\ell\geq2$, which
 implies that $\sigma$ has at least six occurrences of $S$, so
 the length $d$ of $\mathcal{H}$ satisfies $d\geq 6$.
 \end{proof}

\begin{proof}[Proof of Corollary~\ref{cor_formof7}]
 From Equation~\eqref{eqn_thmclassify} in Theorem~\ref{thm_classifyvectors}
 we see that $a_0 = 0$
 and so the given recurrence relation $v_j = a_{j-2}v_{j-1} + v_{j-2}$
 with $j=2$ gives $v_2 = -v_0$.
 
 Suppose $\mathcal{H}$ has associated integers
  $a_0, \ldots, a_{d-1}$ and let $A_0 = [v_0, v_1]$.
 Then
 \begin{equation}
  S^2\overline{A_0^{-1}}T^c\overline{A_0} =_{\STgroup} ST^{a_0}\ldots ST^{a_{d-1}},
 \end{equation}
 implies
 \begin{equation}
  \label{eqn_Jmaxproof1}
  S^2 \overline{A_0^{-1}}T^c =_{\STgroup} ST^{a_0}\ldots ST^{a_{k+1}}
 \end{equation}
 for some $k\in\Z$, since $\overline{A_0}$ starts with $S$.
 By the recurrence relation
 $
  a_i v_{i+1} = v_i + v_{i+2}
 $
 we see
 \begin{equation}
  \label{eqn_Jmaxproof2}
  [v_{k+2}, v_{k+3}] =_{\sltz} A_0 ST^{a_0}\ldots ST^{a_{k+1}}.
 \end{equation}
 Combining Equations~\eqref{eqn_Jmaxproof1} and~\eqref{eqn_Jmaxproof2} yields
 $
  [v_{k+2}, v_{k+3}] =_{\psltz} A_0 S^2 A_0^{-1} T^c
 $
 which implies
 \[
  [v_{k+2}, v_{k+3}] =_{\psltz} T^c =_{\psltz} \matr{1}{c}{0}{1}
 \]
 so $v_{k+2}$ is the required vector.
\end{proof}

\section{Examples}
\label{sec_examples}

\subsection{A representative example of \tex{Theorem~\ref{thm_classifyvectors}}}
\label{sec_representativeex}

Suppose that $\mathcal{H} = (d,c,[\{v_i\}_{i\in\Z}])$ is a minimal
semitoric helix of length $d>4$.
By Corollary~\ref{cor_formof7}, $\mathcal{H}$ is of
type~\eqref{type7} and the representative $\{v_i\}_{i\in\Z}$
can be chosen to satisfy $v_0 = -v_2$.
Then $\mathcal{H}$ is determined by its complexity $c$
and the basis $(v_0,v_1)$ of $\Z^2$.

Let $\mathcal{H}$ have complexity $c=2$ and
\[
 v_0 = \vect{-1}{2}\textrm{ and }v_1 = \vect{-2}{3}.
\]
Next we compute $\mathcal{H}$.
Define $A_0 := [v_0, v_1]$, which means that
$A_0 = ST^2ST^2$ in terms of the generators $S,T$.
Define $A_i = [v_i, v_{i+1}]$ for $i\in\Z$ and notice
that $\det(v_i, v_{i+1})=\det(v_{i+1}, v_{i+2})=1$
implies that there exists some $a_i\in\Z$ so that
\[
[v_{i+1}, v_{i+2}]=[v_{i}, v_{i+1}] \matr{0}{-1}{1}{a_i}
\]
which means that $A_{i+1} = A_iST^{a_i}$ for all $i\in\Z$.
Then,
$
A_d =_{\sltz} A_{d-1}ST^{a_{d-1}} =_{\sltz} A_0ST^{a_0}ST^{a_{d-2}}\ldots ST^{a_{d-1}}
$
and since $\mathcal{H}$ is of length $d$ and complexity $2$,
$ A_d = T^2 A_0$.  Thus
$
 T^2 A_0 =_{\sltz} A_0ST^{a_0}\ldots ST^{a_{d-1}}
$
which implies
\begin{equation}\label{eqn_repexample}
 ST^{a_0}\ldots ST^{a_{d-1}} =_{\sltz} A_0^{-1} T^{2}A_0.
\end{equation}
Our goal is now to recover $a_0, \ldots, a_{d-1}$.
Let $\sigma =_{\STgroup} ST^{a_0}\ldots ST^{a_{d-1}}$.
Lifting Equation~\eqref{eqn_repexample} to the group $G$
yields
\begin{equation}\label{eqn_repexample2}
 \si =_G S^{4k} A_0^{-1}T^2A_0,
\end{equation}
for some choice of $k\in\Z$.  The fact that the vectors
$v_0, \ldots, v_{d-1}$ in the semitoric helix are arranged in
counter-clockwise order forces $k=1$ by
Proposition~\ref{prop_correspondence}.

By substituting for $A_0$ and using the relations of $G$
on Equation~\eqref{eqn_repexample2}
we see that
\[
 \sigma  =_{G} S^4(ST^2ST^2)^{-1}T^2(ST^2ST^2)\\
                             =_{G} S^2T^{-1}ST^2ST^3ST^2ST^2
\]
and by taking the standard form of both sides of this equality we have
\begin{equation}\label{eqn_exampleoverlinesig}
\overline{\sigma} =_{\STgroup} T^{-1}ST^2ST^3ST^2ST^2.
\end{equation}
From this we deduce
$
 W(\sigma) - W(\overline{\sigma}) = \frac{10}{12} - \frac{4}{12} = \frac{1}{2}
$
which means that $\sigma$ can be reduced to
$\overline{\sigma}$ by a substitution of
either
\[S^2 =_{\psltz} I \textrm{ or }ST^{-n}S =_{\psltz} (TST)^n\]
combined with several applications of the blowdown operation,
$STS=_{\psltz} T^{-1}ST^{-1}$.
This comes from 
observation that $S^2 =_{\psltz} I$ and $ST^{-n}S =_{\psltz} (TST)^n$ each
reduce the winding number by $\frac{1}{2}$
and the proof of Theorem~\ref{thm_standardpsltz}, which
uses an algorithm consisting of these three 
reductions to put an element
of $\psltz$ in standard form.

To finish, we recall that $\mathcal{H}$ is minimal, so it does
not admit a blowdown.  This means $\sigma$ may be obtained from $\overline{\sigma}$
by a sequence of blowups followed by either the addition of a $S^2$ term
or the replacement of $(TST)^n$ by $ST^{-n}S$ for some $n>0$.
By examining the form of $\overline{\sigma}$ in Equation~\eqref{eqn_exampleoverlinesig}
we see the requirement that $a_0 = 0$ forces the addition of $S^2$ at the front of the
word and the requirement that $a_i \neq 1$ for all $i$ implies that no blowups may
be performed before adding $S^2$ to the front which forces
$
 \sigma =_{\STgroup} S^2 T^{-1}ST^2ST^3ST^2ST^2.
$
Thus $d = 7$ and
$
 a_0 = 0, a_1 = -1, a_2 = 2, a_3 = 2, a_4 = 3, a_5 = 2, a_6 = 2.
$
Since we were given $v_0$ and $v_1$ these values determine $\mathcal{H}$
by the recurrence relation Equation~\eqref{eqn_virecurrence}.
An image of the first seven vectors in this helix is given
in Figure~\ref{fig_schematic}.

\subsection{Coupled angular momenta}
\label{sec_coupledspin}

Consider 
$\mathbb{S}^2 \times \mathbb{S}^2$ with coordinates 
$(x_1,y_1,z_1,x_2,y_2,z_2)$, where $\mathbb{S}^2$ is the
$2$\--sphere.
The Delzant polygon 
of $\mathbb{S}^2\times \mathbb{S}^2$ endowed with the toric 
integrable system given by $F=(z_1,z_2)$ with any product symplectic form 
is a rectangle, where the length of the sides are determined by the symplectic 
area of each copy of $\mathbb{S}^2$. Its associated fan is formed by the normal
vectors to the faces of the polygon, given by $(1,0), (0,1), (-1,0), (0,-1)$.
If instead we consider the symplectic semitoric manifold with $F=(J,H)$ 
on $\mathbb{S}^2\times\mathbb{S}^2$ with the standard product symplectic
form where
\begin{align*}
 J &= z_1 + \frac{5}{2} z_2\\
 H &= \frac{1}{2} z_1 + \frac{1}{2} (x_1 x_2 + y_1 y_2 + z_1 z_2)
\end{align*}
then we obtain the coupled spin system.
In~\cite{LFPe2016} it is shown that this is indeed a symplectic semitoric
manifold and the authors find the two representatives
of the polygons associated to the system
obtained by
V\~{u} Ng\d{o}c's cutting procedure.
One of these polygons has 
vertices $(-3.5,0),(-1.5,2),(3.5,2),(1.5,0)$. 
The semitoric helix associated to this manifold, which
can be recovered from the polygons as described in
Section~\ref{sec_polytohelix},
 is represented by the three vectors
\[
 v_0 = \vect{0}{1},\, v_1=\vect{-1}{1},\, v_2=\vect{0}{-1}
\]
and
is minimal of type~\eqref{type3}
from Theorem~\ref{thm_classifyvectors} with $k=1$. This is shown
in Figure~\ref{fig_coupledspin}.
Since the system has only one focus-focus point,
the helix can be obtained as the inwards pointing normal
vectors of the semitoric
polygon computed in~\cite{LFPe2016}, where $v_0$ is chosen to
be the inwards pointing normal vector of an edge adjacent to
the fake corner produced by the focus-focus point.
This is because in systems with one focus-focus point
the same toric momentum map that produces
the semitoric polygon also works for the construction of the
semitoric helix, as described in Section~\ref{sec_intrinsicconstruction}
(see Remark~\ref{rmk_consecutivefakes}).

\begin{figure}
 \centering
 \includegraphics[width = 280pt]{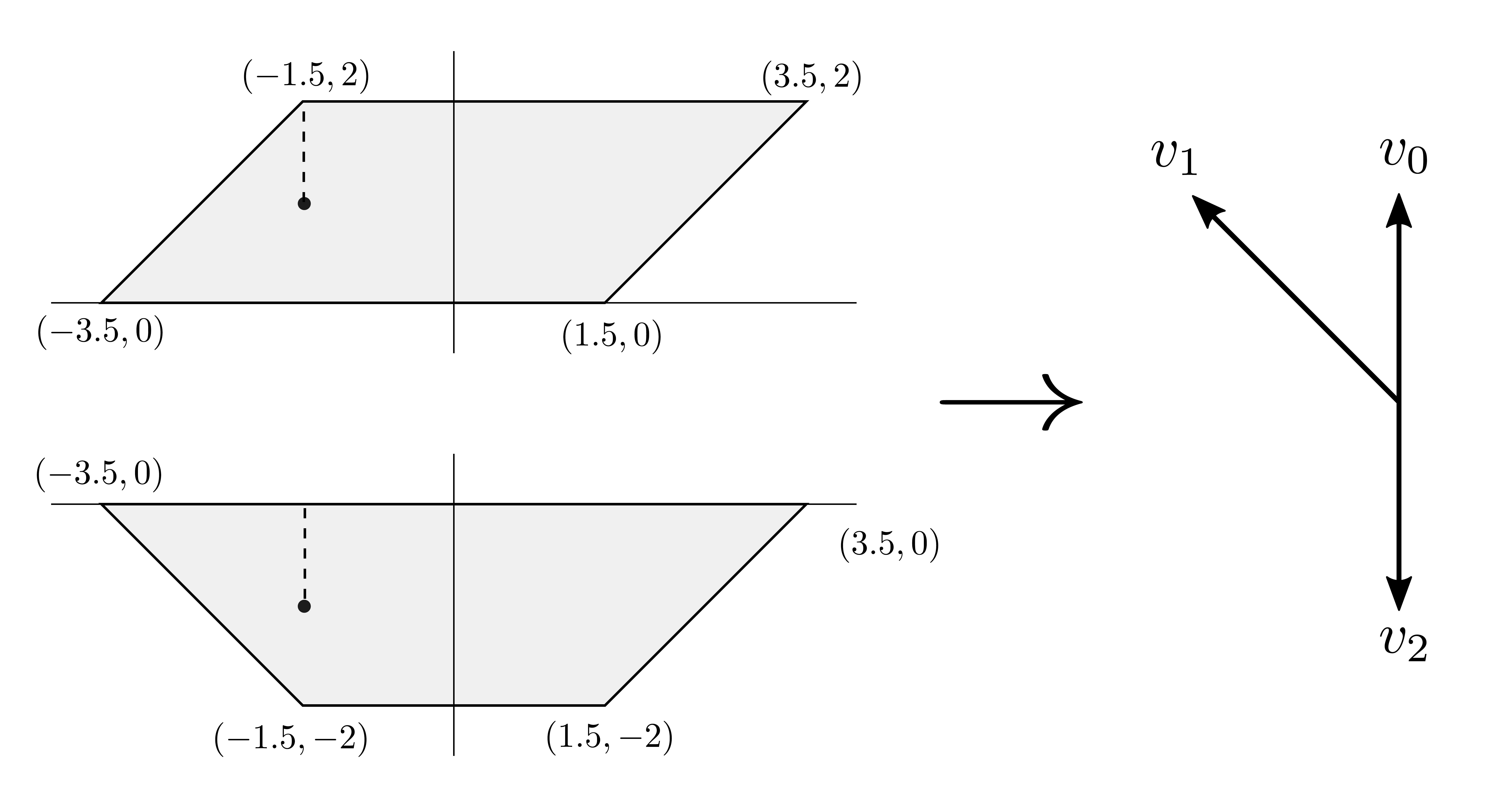}
 \caption{Helix for the coupled spin system.}
 \label{fig_coupledspin}
\end{figure}

\section{Proof of \tex{Theorem~\ref{thm_Jmax}}}
\label{sec_nonminhelices}

Since any helix may be obtained from a minimal
one by a finite sequence of blowups, 
Theorem~\ref{thm_classifyvectors} implies
results about the non-minimal helices as well.

\begin{lemma}
 \label{lem_Jmax}
 Any semitoric helix of complexity $c>2$ includes
 the vector $\vect{1}{0}$ or its negative.
\end{lemma}

\begin{proof}
 We will show that this vector is in every minimal semitoric
 helix of complexity $c>2$, and since every semitoric helix can
 be produced by a sequence of blowups on a minimal semitoric helix
 and blowups do not remove vectors from the helix or change
 the complexity, the result will follow.

 Since $c>2$ the only possible types for minimal models are
 types~\eqref{type4}-\eqref{type7}.  By Theorem~\ref{thm_classifyvectors} we see
 that~\eqref{type4},~\eqref{type5}, and~\eqref{type6} include the required vector.
 Helices of type~\eqref{type7} include the required vector 
 by Corollary~\ref{cor_formof7}.
\end{proof}

\begin{proof}[Proof of Theorem~\ref{thm_Jmax}]
Suppose that $(M,\om,F = (J,H))$ is a compact symplectic semitoric manifold
with $c\geq 2$ focus-focus points and that $J$ achieves
its maximum and minimum at a single point each.
This means that $F(M)$ does not include as its boundary
a vertical line segment, which in turn, since a straightening
map cannot produce a vertical wall, implies that
$\helixmap(M,\om,F)$ does not include a vector on
the $x$\--axis.

Lemma~\ref{lem_Jmax} states that if $c>2$ then
$\helixmap(M,\om,F)$ includes a vector on the $x$\--axis,
so this case cannot occur.
Otherwise, $c=2$.
Since blowups do not change the complexity, Theorem~\ref{thm_classifyvectors}
implies that $\helixmap(M,\om,F)$ can be obtained from
a minimal semitoric helix of either type~\eqref{type2}
or type~\eqref{type7} by a sequence of blowups. By
Corollary~\ref{cor_formof7} any helix of type~\eqref{type7}
includes a vector on the $x$\--axis and thus any blowup
of a helix of type~\eqref{type7} must also include
such a vector. Thus, this case cannot occur.

Suppose that $\helixmap(M,\om,F)$ can be obtained from a minimal
helix of type~\eqref{type2} via a nonzero number of blowups.
There are two distinct blowups which can be performed on a
minimal helix of type~\eqref{type2}, adding either the
vector
\[
\vect{0}{1} + \vect{-1}{-1} = \vect{-1}{0}
\]
or the vector
\[
\vect{-1}{-1} + \vect{2}{1} = \vect{1}{0}.
\]
In either case a vector on the
$x$\--axis is including in the resulting helix, so this
case cannot occur either.
The only remaining case is that $(M,\om,F)$ is minimal
of type~\eqref{type2}.
\end{proof}

\bibliographystyle{amsplain}
\bibliography{st_min_biblio}

\addresses

\end{document}